\documentclass{article}
\usepackage{graphicx} % Required for inserting images
\usepackage{notations}
\usepackage{float}
\usepackage{comment}
\usepackage[acronyms]{glossaries}
\usepackage{subcaption} %For subfigures
\usepackage{amsmath,amssymb,amsthm}%For theorems environments 
\usepackage[a4paper, margin=2.5cm]{geometry} %For geometry
\DeclareUnicodeCharacter{2009}{\,} %For .bib issues
\usepackage[dvipsnames]{xcolor} %Handles colors
\usepackage[colorlinks, linkcolor={blue}, citecolor={red}]{hyperref}
\usepackage{tikz} %For tikz
\usetikzlibrary{positioning, quotes} %For the graphs to draw

\usepackage{bm,dsfont} 
\newcommand{\N}{\mathbb N}
\newcommand{\Z}{\mathbb Z}
\newcommand{\R}{\mathbb R}
\newcommand{\C}{\mathbb C}
\newcommand{\tr}{{\rm Tr}}
\newcommand{\dps}{\displaystyle }

\title{A mathematical analysis of the discretized IPT-DMFT equations}
\author{E. Canc\`es, A. Kirsch and S. Perrin--Roussel}
\makeglossaries %Make the glossary

\newacronym{dmft}{DMFT}{Dynamical Mean-Field Theory}
\newacronym{ipt}{IPT}{Iterated Perturbation Theory}
\newacronym{aim}{AIM}{Anderson Impurity Model}
\newacronym{gft}{GFT}{Generalized Fourier Transform}
\newacronym{car}{CAR}{Canonical Anti-commutation Relations}
\newacronym{gicar}{GICAR}{Gauge Invariant Canonical Anti-commutation Relations}
\newacronym{rhs}{r.h.s.}{right hand side}
\newacronym{lhs}{l.h.s.}{left hand side}
\newacronym{tddft}{TDDFT}{Time-Dependent Density Functionnal Theory}
\newacronym{kms}{KMS}{Kubo-Martin-Schwinger}
\newacronym{acp}{ACP}{Analytical Continuation Problem}
\newacronym{ba}{BU}{Bath Update}

\newacronym{ctqmc}{CTQMC}{Continuous Time Quantum Monte Carlo}

\newacronym{dmet}{DMET}{Density Matrix Embedding Theory}

\graphicspath{{figures}}

\begin{document}

\maketitle

\begin{abstract}
In a previous contribution (E. Canc\`es, A. Kirsch and S. Perrin--Roussel, arXiv:2406.03384), we have proven the existence of a solution to the Dynamical Mean-Field Theory (DMFT) equations under the Iterated Perturbation Theory (IPT-DMFT) approximation.
In view of numerical simulations, these equations need to be discretized. In this article, we are interested in a discretization of the \acrshort{ipt}-\acrshort{dmft} functional equations, based on the restriction of the hybridization function and local self-energy to a finite number of points in the upper half-plane $\left(i\MatsubaraFrequency_n\right)_{n \in \IntSubSet{0}{\MatsubaraFrequencyCutoff}}$, where $\MatsubaraFrequency_n=(2n+1)\pi / \StatisticalTemperature$ is the $n$-th Matsubara frequency and $\MatsubaraFrequencyCutoff \in \Integers$.
We first prove the existence of solutions to the discretized equations in some parameter range depending on $\MatsubaraFrequencyCutoff$.
We then prove uniqueness for a smaller range of parameters. We also study more in depth the case of bipartite systems exhibiting particle-hole symmetry. In this case, the discretized IPT-DMFT equations have purely imaginary solutions, which can be obtained by solving a real algebraic system of $(\MatsubaraFrequencyCutoff+1)$ equations with  $(\MatsubaraFrequencyCutoff+1)$ variables. We provide a complete characterization of the solutions for $\MatsubaraFrequencyCutoff=0$ and some results for  $\MatsubaraFrequencyCutoff=1$ in the simple case of the Hubbard dimer. We finally present some numerical simulations on the Hubbard dimer. First, using the \texttt{TRIQS} Python library \cite{parcollet_triqs_2015}, we describe how a conductor-to-insulator transition occurs in the Matsubara frequency discretized IPT-DMFT model in the small temperature regime. Finally, using a homemade Julia code allowing arbitrary finite-arithmetic precision, we show that for some parameters $(U,T)$ of the considered Hubbard model, this discretization method may provide values of the Green's function at the lowest Matsubara frequencies which cannot be interpolated by a Pick function.
\end{abstract}

\tableofcontents

\section{Introduction}\label{sec:introduction}

In the realm of condensed matter physics, the \acrfull{dmft} \cite{georges_hubbard_1992,georges_dynamical_1996,metzner_correlated_1989,zhang_mott_1993} is an approximate method designed to tackle the quantum many-body problem associated with complex materials, in which interaction-driven effects can arise such as high critical temperature superconductivity or magnetism.
This method was first applied to the celebrated Hubbard model \cite{hubbard_electron_1963,kanamori_electron_1963,gutzwiller_effect_1963,pariser_semiempirical_1953,pople_electron_1953} (for which much remains unknown apart from specific geometries \cite{lieb_hubbard_2004}), and then  to periodic materials, in combination with \emph{first-principle} methods such as Density Functional Theory (leading to the DFT+DMFT method \cite{held_mott-hubbard_2001}) or the GW approximation (leading to the GW+DMFT method \cite{sun_extended_2002}).

The purpose of DMFT is to provide an approximation of the \emph{one-body Green's function}. The reader unfamiliar with this formalism is referred to e.g. \cite{cances_mathematical_2016,cances_mathematical_2024} for a mathematically-oriented presentation of Green's function methods in the context of electronic structure calculation in general and DMFT in particular.

More precisely, the success of the \acrfull{dmft} in the strongly-correlated condensed-matter community lies in its ability to provide good approximations of the local Green's function and local self-energy of very large systems \cite{georges_dynamical_1996}.

The computational bottleneck of  a \acrshort{dmft} calculation is the \emph{impurity solver}. To fix the ideas, consider a Hubbard model with $L$ sites. At each macro-iteration of the DMFT iteration loop, $L$ auxiliary Anderson Impurity Models (AIMs) are constructed and have to be solved using an impurity solver. Most of practical implementations of \acrshort{dmft} are concerned with the \emph{translation-invariant} setting, for which the $L$ AIMs are in fact identical. Still, the impurity solver is the practical bottleneck (even ground-state computations are known to be very expensive \cite{bravyi_complexity_2017}): accurate computations can be performed using e.g. \acrfull{ctqmc} solvers \cite{gull_continuous-time_2011,rubtsov_continuous-time_2005,werner_continuous-time_2006,werner_hybridization_2006,li_interaction-expansion_2022,seth_triqscthyb_2016}, which are supposedly exact up to statistical noise.
Nevertheless, their computational cost can be too high for specific applications, such as moir\'e heterostructures \cite{tan_doping_2023}, so that computations in this setting must be performed using \emph{approximate solvers}. One of the computationally cheapest (and least accurate) solvers is the Iterated Perturbative Theory (IPT) solver~\cite{georges_hubbard_1992,zhang_mott_1993,georges_dynamical_1996}.
We have studied its analytical properties in \cite{cances_mathematical_2024}, in the \emph{single-site translation-invarariant paramagnetic} setting, to which we will also stick in this article. In this setting, the  \acrshort{ipt}-\acrshort{dmft} equations read
\begin{align}
\forall z \in \UpperHalfPlane, \quad \Hybridization(z)&=W^T\left(z-\NIHamiltonian_\perp -\SelfEnergy(z)\right)^{-1}W, \label{eq:BathUpdateUpperHalfPlane} \\
\SelfEnergy&=U^2 \IPTmap_\StatisticalTemperature(\Hybridization). \label{eq:IPTUpperHalfPlane}
\end{align}
The unknown are the {\em hybridization function} $\Hybridization : \UpperHalfPlane \to -\overline{\UpperHalfPlane}$ and local self-energy $\SelfEnergy : \UpperHalfPlane \to -\overline{\UpperHalfPlane}$, where $\UpperHalfPlane:=\{z \in \C, \; \Im(z) > 0\}$ is the (open) complex upper half-plane, and $\overline{\UpperHalfPlane}=\{z \in \C, \; \Im(z) \ge 0\}$ its closure. The parameters of the model are the graph $\HubbardGraph=(\HubbardVertices,\HubbardEdges)$ of the translation-invariant Hubbard model ($\HubbardVertices$ is the set of sites, identified with the vertices of the graph, $\HubbardEdges$ the set of edges connecting sites between which hoping is allowed, and $L=\Cardinal{\HubbardVertices}$ is the number of sites), the hopping parameter $\HoppingMatrix[] \in \RealNumbers$, the on-site repulsion $\OnSiteRepulsion[] \in \RealNumbers$, and the inverse temperature $\StatisticalTemperature$. In the language of graph theory, the translation-invariance of the Hubbard model means that the graph $\HubbardGraph$ is vertex transitive.

 The vector  
$W \in \RealNumbers^{L-1}$ and the real symmetric matrix $\NIHamiltonian_\perp \in \R^{(L-1) \times (L-1)}_{\rm sym}$ are obtained from  $\HubbardGraph$ and $T$ through the relation 
\begin{equation}\label{eq:DefinitionBUMatrices}
\HoppingMatrix[] \; \AdjacencyMatrix= \left(\begin{matrix}
0 & W^T \\
W & \NIHamiltonian_\perp
\end{matrix}\right), \nonumber
\end{equation}
where $\AdjacencyMatrix$ denotes the adjacency matrix of the (translation invariant) Hubbard graph $\HubbardGraph$. Lastly, $F_\StatisticalTemperature^{\rm IPT}$ is the IPT solver for inverse temperature $\StatisticalTemperature$. Its definition will be recalled in the next section. In \cite{cances_mathematical_2024}, we proved that Eqs~\eqref{eq:BathUpdateUpperHalfPlane}--\eqref{eq:IPTUpperHalfPlane} have a solution $(\Hybridization,\SelfEnergy)$ for any set of parameters $\HubbardGraph$, $T$, $U$, $\StatisticalTemperature$.

\medskip

A key ingredient in both the mathematical analysis and the numerical resolution of the IPT-DMFT equations is that both $\Hybridization$ and $\SelfEnergy$ are analytic in $\UpperHalfPlane$. It follows that $\Hybridization$ and $\SelfEnergy$ are negatives of Pick functions (a Pick function is an analytic function $f:\UpperHalfPlane \to \ComplexNumbers$ such that $\Im (f(z)) \ge 0$ for all $z \in \UpperHalfPlane$). A popular numerical method consists in seeking approximations $\HybridizationVector=\left(\Hybridization_n\right)_{n \in \IntSubSet{0}{\MatsubaraFrequencyCutoff}}\in -\overline{\UpperHalfPlane}^{\MatsubaraFrequencyCutoff+1}$ and $\SelfEnergyVector=\left(\SelfEnergy_n\right)_{n \in \IntSubSet{0}{\MatsubaraFrequencyCutoff}}\in -\overline{\UpperHalfPlane}^{\MatsubaraFrequencyCutoff+1}$ of the values $\left(\Hybridization(i\omega_n)\right)_{n \in \IntSubSet{0}{\MatsubaraFrequencyCutoff}}$ and $\left(\SelfEnergy(i\omega_n)\right)_{n \in \IntSubSet{0}{\MatsubaraFrequencyCutoff}}$ of the hybridization function and local self-energy at the lowest $(\MatsubaraFrequencyCutoff+1)$ imaginary Matsubara frequencies of the upper half-plane. Recall that the Matsubara frequencies are defined by the formula $\MatsubaraFrequency_n=(2n+1)\frac{\pi}{\StatisticalTemperature}$, $n \in \Z$, and depend on the temperature, even if this dependency is omitted in the notation. The unknown of the Matsubara frequencies (MaF) discretization of the IPT-DMFT equations is then the pair $(\HybridizationVector,\SelfEnergyVector)\in\LowerHalfPlaneVectors\times  \LowerHalfPlaneVectors$.
As Equation~\eqref{eq:BathUpdateUpperHalfPlane} is local, its discretized counterpart is simply 
\begin{equation}\label{eq:discretized_Eq1}
\forall n \in \IntSubSet{0}{\MatsubaraFrequencyCutoff}, \quad \Hybridization_n=W^T\left(z-\NIHamiltonian_\perp -\SelfEnergy_n\right)^{-1}W.
\end{equation}
The discretized counterpart 
\begin{equation}\label{eq:discretized_Eq2}
\SelfEnergyVector =U^2 \IPTmapDiscretized(\HybridizationVector)
\end{equation}
of the nonlocal equation~\eqref{eq:IPTUpperHalfPlane} will be discussed below. Let us only mention at this point that it is a rational function in the $\Hybridization_n$'s and $\overline{\Hybridization_n}$'s so that \eqref{eq:discretized_Eq1}--\eqref{eq:discretized_Eq2} can be reformulated as a set of algebraic equations in the variables $(\Re(\Hybridization_n),\Im(\Hybridization_n))_{0 \le n \le \MatsubaraFrequencyCutoff}$.

\medskip

In this article, we study the set of discretized equations \eqref{eq:discretized_Eq1}--\eqref{eq:discretized_Eq2} from a theoretical and numerical viewpoint. In Section~\ref{sec:MFDIPTDMFTIntroduction}, we recall the physical meaning and structure of the continuous IPT-DMFT equations \eqref{eq:BathUpdateUpperHalfPlane}--\eqref{eq:IPTUpperHalfPlane}  and of the Matsubara frequency discretization of these equations. We also introduce a dimensionless version of the discretized equations \eqref{eq:discretized_Eq1}--\eqref{eq:discretized_Eq2} suitable for mathematical analysis. In Section~\ref{sec:MainResults}, we present our main theoretical results. The first one is concerned with the existence of physically admissible solutions, that are solutions with nonpositive imaginary parts. While the original IPT functional $\IPTmap_{\StatisticalTemperature}$ maps the set of negative of Pick functions into itself, its discretizated version $\IPTmap_{\StatisticalTemperature,\MatsubaraFrequencyCutoff}$ does not map $\LowerHalfPlaneVectors$ into itself.
Because of this possible limitation of the MaF discretization scheme of \eqref{eq:BathUpdateUpperHalfPlane}--\eqref{eq:IPTUpperHalfPlane}, we were only able to prove the existence of physically admissible solutions to \eqref{eq:discretized_Eq1}--\eqref{eq:discretized_Eq2} in some parameter range, when $\StatisticalTemperature T\sqrt{{\rm deg}(\HubbardGraph)}$ is smaller than a critical value depending on $\MatsubaraFrequencyCutoff$. Recall that ${\rm deg}(\HubbardGraph) \in \N$ denotes the degree of the vertex-transitive Hubbard graph $\HubbardGraph$. On the other hand, we are able to prove a uniqueness result in some parameter range, when $\StatisticalTemperature^2 TU\sqrt{{\rm deg}(\HubbardGraph)}$ is smaller than a critical value depending on $\MatsubaraFrequencyCutoff$. In Section~\ref{sec:BipartiteSystems}, we focus on the special case of bipartite systems at half filling. Such systems exhibit particle-hole symmetry, which allows one to seek solutions on the negative imaginary axis. Equations \eqref{eq:discretized_Eq1}--\eqref{eq:discretized_Eq2} can then be reformulated as a problem of real algebraic geometry: seek the solutions in $(0,1]^{\MatsubaraFrequencyCutoff+1}$ of a set of $(\MatsubaraFrequencyCutoff+1)$ real, sparse, polynomial equations in $(\MatsubaraFrequencyCutoff+1)$ real variables with degree $(6M+4)$, where $M$ is an integer depending on the adjacency matrix of the Hubbard graph. As a matter of illustration, we consider the Hubbard dimer (for which $L=2$ and $M=0$), and study the corresponding algebraic equations for $\MatsubaraFrequencyCutoff=0$ and $\MatsubaraFrequencyCutoff=1$. In Section~\ref{sec:IterativeSchemeMottTransiton}, we provide some numerical simulation results for the Hubbard dimer for larger values of $\MatsubaraFrequencyCutoff$. The purpose of this section is by no means to obtain new physical insights on the Hubbard model, but to illustrate on a simple example the properties of the MaF discretization method and the associated DMFT iteration scheme (convergence rate, multiple solutions, etc.). Moreover, it is highlited in Section \ref{sec:PickCriterion} that for some parameter range $(T,U)$ of the Hubbard model, the numerically obtained values of the Green's function at the lowest $(\MatsubaraFrequencyCutoff+1)$ positive Matsubara frequencies do not satisfy the Pick criterion, which means that they cannot be exactly interpolated by a Pick function.

\section{Matsubara frequency discretization of IPT-DMFT equations} \label{sec:MFDIPTDMFTIntroduction}

\subsection{Derivation of the IPT-DMFT equations}

To make this article as self-contained as possible, let us start with a concise mathematical introduction to DMFT and the IPT approximation, and recall the derivation of Equations~\eqref{eq:BathUpdateUpperHalfPlane}-\eqref{eq:IPTUpperHalfPlane}. If $\mathcal X$ and $\mathcal Y$ are Hilbert spaces, we denote by $\mathcal L(\mathcal X;\mathcal Y)$ the space of bounded linear operators from $\mathcal X$ to $\mathcal Y$, by $\mathcal A(\mathcal X;\mathcal Y)$ the space of bounded antilinear operators from $\mathcal X$ to $\mathcal Y$, and we use the usual shorthand notation $\mathcal L(\mathcal X):=\mathcal L(\mathcal X;\mathcal X)$. Lastly, we denote by $\mathcal S(\mathcal X)$ the space of self-adjoint operators on $\mathcal X$.

\medskip

Consider an isolated fermionic quantum system with one-particle state space $\OneParticleSpace$, which we assume finite-dimensional for simplicity. We denote respectively by 
$$
{\rm Fock}(\OneParticleSpace):=\dps \bigoplus_{N=0}^{\dim(\OneParticleSpace)} \bigwedge^N \OneParticleSpace, \qquad \widehat a_\bullet^\dagger \in \mathcal L(\OneParticleSpace;\mathcal L({\rm Fock}(\OneParticleSpace))), \qquad \widehat a_\bullet \in \mathcal A(\OneParticleSpace;\mathcal L({\rm Fock}(\OneParticleSpace))), \qquad \widehat H \in \mathcal S({\rm Fock}(\OneParticleSpace)), 
$$
 the Fock space, the creation and annihilation operators in the Schr\"odinger picture, and the Hamiltonian of the system in second quantization formalism. Recall that the creation and annihilation operators  satisfy the \acrfull{car}:
$$
\forall \phi,\phi' \in \OneParticleSpace, \quad \{\widehat a_\phi^\dagger,\widehat a_{\phi'}^\dagger\} = 0, \quad   \{\widehat a_\phi,\widehat a_{\phi'}\} = 0, \quad  \{\widehat a_\phi,\widehat a_{\phi'}^\dagger\} =\langle \phi,\phi'\rangle_{\OneParticleSpace}.
$$
We assume here that $\widehat H$ is particle-number conserving. 

\medskip

The time-ordered one-body Green's function associated to a quantum system with one-particle state space $\OneParticleSpace$ (assume $\dim(\OneParticleSpace) < \infty$ for simplicity) and an equilibrium state $\widehat \Gamma$ is the bounded function $G^{\rm TO} \in L^\infty(\R_t;\mathcal L(\OneParticleSpace))$ of the time variable $t$ with values in ${\mathcal L}(\OneParticleSpace)$ defined by
$$
\forall \phi,\phi' \in \OneParticleSpace, \quad \langle \phi, iG^{\rm TO}(t)\phi'\rangle_{\OneParticleSpace}:=\tr\left(\widehat \Gamma \; \mathcal T\left( \widehat a_\phi, \widehat a^\dagger_{\phi'}\right)(t,0)  \right),
$$
where $\mathcal T$ is the time-ordering operator and $\widehat a_\bullet(\cdot)^\dagger$, $\widehat a_\bullet(\cdot)$ the creation/annihilation operators in Heisenberg picture, that is 
$$
\langle \phi, iG^{\rm TO}(t)\phi'\rangle_{\OneParticleSpace}=\left| \begin{array}{ll} \tr\left(\widehat \Gamma \;  \widehat a_\phi(t) \widehat a^\dagger_{\phi'}(0) \right), &  \mbox{ if } t \geq 0, \\ - \tr\left(\widehat \Gamma \; \widehat a^\dagger_{\phi'}(0)  \widehat a_\phi(t) \right), &  \mbox{ if } t < 0, \end{array} \right. \quad \mbox{with} \quad \widehat a_\phi(t) = e^{it\widehat H} \widehat a_\phi e^{-it\widehat H}, \quad  \widehat a^\dagger_{\phi'}(0) = \widehat a^\dagger_{\phi'}.
$$
By definition, an equilibrium state is a steady state of the quantum Liouville equation $i\partial_t\widehat \Gamma(t) = [\widehat H,\widehat \Gamma(t)]$, that is an operator $\widehat \Gamma \in \mathcal L({\rm Fock}(\OneParticleSpace))$ satisfying the mixed-state representability conditions $\widehat \Gamma^\dagger=\widehat \Gamma$, $\widehat \Gamma  \ge 0$, and $\tr(\widehat \Gamma) =1$, and commuting with the Hamiltonian, i.e. such that $[\widehat H,\widehat \Gamma]=0$. In practice, $\widehat\Gamma$ is  often an $N$-particle ground state (microcanonical equilibrium state), a Gibbs state in the $N$-particle sector of the Fock space (canonical equilibrium state), or a Gibbs state for the modified Hamiltonian $\mathcal K = \widehat H-\ChemicalPotential \widehat N$ (grand canonical equilibrium state), where $\ChemicalPotential$ is a chemical potential and $\widehat N$ the number operator. Note that in all these cases, $\widehat \Gamma$ commutes with $\widehat N$, an assumption we make from now on. Since $\widehat H$, $\widehat \Gamma$ and $\widehat N$ pairwise commute (recall that $\widehat H$ is assumed to be particle-number conserving), they can be diagonalized in the same orthonormal basis $\{ \Psi^N_n \ ; \ 0 \le N \le \dim(\OneParticleSpace), \, 1 \le n \le \tiny{\begin{pmatrix}\dim(\OneParticleSpace)\\ N\end{pmatrix}}\}$ of the Fock space:
$$
\widehat N \Psi^N_n = N \Psi^N_n, \quad \widehat H \Psi^N_n = E^N_n \Psi^N_n, \quad \widehat \Gamma \Psi^N_n = f^N_n \Psi^N_n,
$$
where $E^N_n \in \R$ is the energy of the state $\Psi^N_n \in \dps\bigwedge^N \OneParticleSpace \subset {\rm Fock}(\OneParticleSpace)$ and $0 \le f^N_n \le 1$ its occupation number in the equilibrium state $\widehat \Gamma$. With this notation $\Psi^0_0$ is the vacuum of the Fock space ${\rm Fock}(\OneParticleSpace)$.

\medskip

In many applications, it is more relevant from a physical point of view to work in the frequency domain and focus instead on the time Fourier transform $G^{\rm TO}_{\rm freq} \in {\mathcal S}'(\R_\omega;\mathcal L(\OneParticleSpace))$ of $G^{\rm TO}$. However, this object has singularities (it is in general an order-$1$ tempered distribution) and it is more convenient to consider the Generalized Fourier Transform $G: \UpperHalfPlane \to \mathcal L(\OneParticleSpace)$ of $G^{\rm TO}$  in the upper half-plane. It indeed turns out that the function $G$ has very interesting mathematical properties. In particular, it is the negative of an operator-valued Pick function \cite{nevanlinna_uber_1919,pick_uber_1915} which means that it is analytic on $\UpperHalfPlane$  and satisfies 
$\Im(-G(z)):=-\frac{G(z)-G(z)^\dagger}{2i} \ge 0$ for all $z \in \UpperHalfPlane$. It readily follows from basic properties of analytic functions that for any non-constant Pick function $f$, $\Im(f(z))>0$ for all $z \in \UpperHalfPlane$. This implies that for all $z \in \UpperHalfPlane$, $\Im(-G(z))$ is in fact positive definite, so that $G(z)$ is invertible. 

\medskip

In addition, $G(z)$ vanishes when $\Im(z) \to +\infty$, so that it follows from the Nevanlinna-Riesz representation theorem \cite{gesztesy_matrixvalued_2000} that
$$
\forall z \in \UpperHalfPlane, \quad \GreensFunction(z) = \int_\RealNumbers \frac{d\mathcal A(\eps)}{z-\eps}, 
$$
where $\mathcal A$ is an operator-valued Borel measure on $\R$ such that
$$
\forall B \in \mathcal B(\R), \quad \mathcal A(B) \in \mathcal S(\OneParticleSpace), \quad 0 \le \mathcal A(B) \le 1, \quad \mathcal A(\R)=I_{\OneParticleSpace},
$$
where $\mathcal B(\R)$ is the set of Borel subsets of $\R$.
The operator-valued measure $\mathcal A$ is called the spectral function of the system in the equilibrium state $\widehat \Gamma$ under consideration. Another useful expression of $G$ is its K\"allen--Lehmann representation, obtained by expanding the operators $\widehat \Gamma$ and $\widehat H$ in the orthonormal basis $(\Psi^N_n)$ diagonalizing both of them. We indeed have
$$
\langle \phi,G(z)\phi' \rangle_{\OneParticleSpace}= \sum_{N=0}^{\tiny{\dim(\OneParticleSpace)-1}} 
\sum_{n=0}^{\tiny{\begin{pmatrix}\dim(\mathcal{H})\\ N\end{pmatrix}}}
\sum_{n'=0}^{\tiny{\begin{pmatrix}\dim(\mathcal{H})\\ N+1\end{pmatrix}}}
 \frac{f^N_n + f^{N+1}_{n'}}{z-(E^{N+1}_{n'}-E^N_n)} \langle \Psi^N_n,\widehat a_\phi \Psi^{N+1}_{n'}\rangle_{{\rm Fock}(\OneParticleSpace)} \langle \Psi^{N+1}_{n'},\widehat a_{\phi'}^\dagger \Psi^N_n \rangle_{{\rm Fock}(\OneParticleSpace)}.
$$
The Källen--Lehmann representation reveals that $G$ is a rational function of $z$, the poles of which are some of the excitation energies of the system. More precisely, if $\Gamma$ is the $N$-particle ground state, the poles of $G$ are the excitation energies $E^{N+1}_n-E^N_0$ (one-particle affinities) and $E^N_0-E^{N_1}_n$ (one-particle ionization energies). The same holds for the canonical Gibbs state in the $N$-particle sector. 

\medskip

In the special case of a non-interacting system, i.e. if the Hamiltonian is given by $\widehat H^0 = d\Gamma(H^0)$ for some hermitian operator $H^0$ on $\OneParticleSpace$, the Green's function and the spectral function are given by
$$
G^0(z) = (z-H^0)^{-1}, \quad \mathcal A_0(\bullet) = \Pi^{H^0}_\bullet, 
$$
where $\Pi^{H^0}_\bullet$ is the spectral projector of $H^0$. In particular, $G^0$ is independent of $\widehat \Gamma$ (in contrast with the interacting case) and is simply the restriction to $\UpperHalfPlane$ of the resolvent of $H_0$.

\medskip

Let us now consider a Hamiltonian $\widehat H$ on ${\rm Fock}(\OneParticleSpace)$ of the form
$$
\widehat H = d\Gamma(H^0) + \widehat H^{\rm I} \qquad \mbox{(non-interacting part + interacting part)},
$$
and $\GreensFunction:\UpperHalfPlane \to \mathcal L(\OneParticleSpace)$ the Green's function of $\widehat H$ for some equilibrium state $\Gamma$. The self-energy is the function $\SelfEnergy_{\rm g}:\UpperHalfPlane \to \mathcal L(\OneParticleSpace)$ (the subscript $g$ stands for global, in contrast with local, see below) defined by
$$
\SelfEnergy_{\rm g}(z):=G^0(z)^{-1} - G(z)^{-1} \qquad \mbox{or} \qquad G(z)=\left(z-\left(H^0+\SelfEnergy_{\rm g}(z)\right) \right)^{-1}.
$$
Note that $-\SelfEnergy$ also is an operator-valued Pick function. For later purposes, we introduce the functional $g$
$$
[g(\SelfEnergy)](z) = (z-(H^0+\SelfEnergy(z))^{-1} \in \mathcal L(\OneParticleSpace)
$$
which is well-defined on the set of negatives of operator-valued Pick functions from $\UpperHalfPlane$ to $\mathcal L(\OneParticleSpace)$ and such that $G=g(\SelfEnergy_{\rm g})$.

\medskip

The Anderson impurity model (AIM) is a fermionic quantum many-body model with a special structure. The one-particle state space is of the form
$$
\OneParticleSpace_{\mathrm{AIM}} = \OneParticleSpace_{\rm imp} \oplus  \OneParticleSpace_{\rm bath},
$$
with $\dim(\OneParticleSpace_{\rm imp})$ ``small'' and $\dim(\OneParticleSpace_{\rm bath})$ potentially large or even infinite, and the interactions are localized in the impurity component. As a consequence, the AIM Hamiltonian is of the form 
$$
\widehat H_{\rm AIM} = d\Gamma(H^0_{\rm AIM}) + \widehat H_{\rm imp}^{\rm I} \otimes 1_{{\rm Fock}(\OneParticleSpace{\rm bath})}, \qquad 
H^0_{\rm AIM} = \left( \begin{array}{cc} H^0_{\rm i} & H^0_{\rm ib} \\ {H^0_{\rm ib}}^\dagger  & H^0_{\rm b} \end{array} \right)
$$
(recall that with ${\rm Fock}(\OneParticleSpace_{\rm AIM}) = {\rm Fock}(\OneParticleSpace_{\rm imp}) \otimes  {\rm Fock}(\OneParticleSpace_{\rm bath})$).
It can be shown for finite-dimensional bath~\cite{lin_sparsity_2020,cances_mathematical_2024} that the AIM self-energy for the Gibbs state $Z^{-1}e^{-\StatisticalTemperature(\widehat H_{\rm AIM}-\mu \widehat N_{\rm frag})}$ is of the form
$$
\SelfEnergy_{\StatisticalTemperature,\mu}(z)= \left( \begin{array}{cc} \SelfEnergy_{{\rm i},\StatisticalTemperature,\mu}(z) & 0 \\ 0  & 0 \end{array} \right) \qquad \mbox{and} \qquad \SelfEnergy_{{\rm i},\StatisticalTemperature,\mu}=\SelfEnergy^{\rm AIM}_{{\rm imp},\StatisticalTemperature,\mu}\left(\OneParticleSpace_{\rm imp},H^0_{\rm i},\Hybridization,\widehat H_{\rm imp}^{\rm I}\right),
$$
where $\Hybridization:\UpperHalfPlane \to \mathcal L(\OneParticleSpace_{\rm imp})$ is the {\em hybridization function}
$$
\Hybridization(z):= H^0_{\rm ib} (z-H^0_{\rm b})^{-1} {H^0_{\rm ib}}^\dagger.
$$
Note that $\Hybridization$ is the negative of an operator-valued Pick function as well. It is assumed in the physics literature that this also holds true for infinite-dimensional baths. For later purposes, we introduce the functional $g_{\rm imp}^{\rm AIM}$ defined by
$$
[g_{\rm imp}^{\rm AIM}(\OneParticleSpace_{\rm imp},H^0_{\rm i},\Hybridization,\SelfEnergy_{\rm i})](z) \!\!:=\!\! \!\!\ \left(\! z\!- \!\left( H^0_{\rm i} + \Hybridization(z) + \SelfEnergy_{\rm i}(z) \right) \right)^{-1} \!\!\!\!\!\!= \!\! \left( \! z \!- \!\left( \!\!H_{\rm AIM}^0  \! + \!\begin{pmatrix}\SelfEnergy_{\rm i}(z) & 0 \\ 0  & 0\end{pmatrix} \right)\right)^{-1} \! \bigg|_{\OneParticleSpace{\rm imp}},
$$
which is well-defined for $H^0_{\rm i} \in \mathcal S(\OneParticleSpace_{\rm imp})$ and $\SelfEnergy_{\rm i}$, $\Hybridization$ negatives of operator-valued Pick functions from $\UpperHalfPlane$ to $\mathcal L(\OneParticleSpace_{\rm imp})$. The impurity block of the exact Green's function of the AIM for the considered Gibbs state is given by
$$
G_{\rm AIM}|_{\OneParticleSpace{\rm imp}}=g_{\rm imp}^{\rm AIM}(\OneParticleSpace_{\rm imp},H^0_{\rm i},\Hybridization,\SelfEnergy_{\rm i}).
$$

\medskip

DMFT is a quantum embedding method in which the system is decomposed in non-overlapping fragments, each of them being seen as an impurity embedded in a bath accounting in the presence of the other fragments. In mathematical terms, the system state space $\OneParticleSpace$ is decomposed as a direct sum of pairwise orthogonal impurity subspaces:
$$
\OneParticleSpace = \OneParticleSpace_{\rm imp}^1 \oplus \cdots \oplus \OneParticleSpace_{\rm imp}^{M}
$$
and we denote by 
$$
H^0 = \left( \begin{array}{cc} H^0_{j} & H^0_{j\bar j} \\ {H^0_{j\bar j}}^\dagger & H^0_{\bar j} \end{array} \right), \qquad \widehat H^{\rm I} = \widehat H^{\rm I}_{j} + \widehat H^{\rm I}_{j\bar j} + \widehat H^{\rm I}_{\bar j} 
$$
the block representation of $H^0$ according to the decomposition $\OneParticleSpace = \OneParticleSpace^j_{\rm frag} \oplus (\OneParticleSpace^j_{\rm frag})^\perp$ and the corresponding decomposition of the interaction Hamiltonian $\widehat H^{\rm I}$.
 
\medskip

As other quantum embedding methods, such as DMET, DMFT can be seen as a fixed-point problem on a set of low-level descriptors of the state of the whole system. In DMFT, this set can be chosen as a set of block-diagonal approximations of the global self-energy
$$
\mathfrak S_{\rm DMFT}:=\left\{\SelfEnergy = \left(\begin{matrix}
\SelfEnergy_1 & 0 &\cdots& 0\\
0& \SelfEnergy_2&\cdots&0 \\
\vdots&\vdots&\ddots&\vdots \\
0&0&\cdots& \SelfEnergy_{M}
\end{matrix}\right),  \; \SelfEnergy_j : \UpperHalfPlane \to \mathcal L(\OneParticleSpace^j_{\rm imp}), -\SelfEnergy_j  \mbox{ operator-valued Pick function} \right\}
$$
The fixed-point iteration goes in two steps:
\begin{enumerate}
\item AIM construction: given $\SelfEnergy^{\rm in} \in \mathfrak S_{\rm DMFT}$, for each $j$, construct an AIM such that
$$
g\left(\SelfEnergy^{\rm in} \right)\big|_{\OneParticleSpace_{\rm imp}^j} = 
g_{\rm imp}^{\rm AIM}\left(\OneParticleSpace{\rm imp}^j,H^0_j,\Hybridization_j^{\rm out},\SelfEnergy_j^{\rm in}  \right)
$$
This does not uniquely define the AIM, but this does uniquely define $\Hybridization_j^{\rm out}$:
$$
\Hybridization_j^{\rm out}(z) = H^0_{j\bar j} \left( z - \left( H^0_{\bar j} + \bigoplus_{k \neq j} \SelfEnergy^{\rm in}_k(z) \right) \right)^{-1} {H^0_{j\bar j}}^\dagger;
$$
In physical terms, the hybridization function $\Hybridization_j$ encapsulates the effects of the environment of the $j$-th cluster and allows one to consider the latter as an impurity embedded in a bath.
\item Impurity solver: for each of these AIMs, compute an updated $\SelfEnergy^{\rm out} \in \mathfrak S_{\rm DMFT}$ by
$$
\SelfEnergy_j^{\rm out}:=\SelfEnergy^{\rm AIM}_{{\rm imp},\StatisticalTemperature,\mu}\left(\OneParticleSpace^j_{\rm imp},H^0_{j},\Hybridization_j^{\rm out},\widehat H^{\rm I}_{j}\right).
$$
In physical terms, for each cluster, compute the self-energy of the associated AIM, and concatenate the so-obtained solutions in the updated low-level descriptor $\SelfEnergy^{\rm out}$.
\end{enumerate}

We now restrict ourselves to a translation-invariant Hubbard model with $L < \infty$ sites, single-site impurities, and paramagnetic setting at half-filling. In this case 
$$
\quad M=L, \quad \OneParticleSpace^1_{\rm imp} \cong \cdots \cong \OneParticleSpace^L_{\rm imp} \cong \C^2, \quad \Hybridization_1=\cdots= \Hybridization_L=\Hybridization {\mathds 1}_2, \quad \SelfEnergy_1=\cdots= \SelfEnergy_L=\SelfEnergy {\mathds 1}_2. 
$$
In addition, we assume that the AIMs are solved by Iterative Perturbation Theory:
$$
[\SelfEnergy^{\rm AIM-IPT}_{{\rm imp},\StatisticalTemperature}(\Hybridization,U)](z) = U^2 [\IPTmap_\StatisticalTemperature(\Hybridization)](z) \mathds{1}_2,
$$
where $\IPTmap_\StatisticalTemperature$ is derived by perturbation theory at second-order in the on-site interaction parameter $U$. As we will see on Eq.~\eqref{eq:Sigma_IPT_MF} and the comments below, $\IPTmap_\StatisticalTemperature(\Hybridization)$ can be conveniently characterized by its values on the imaginary Matsubara frequencies.

\begin{figure}[h]
\centering

\begin{subfigure}{0.3\textwidth}
\centering
\includegraphics[width=\textwidth]{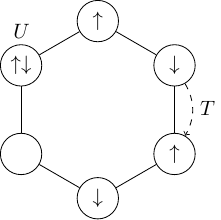}    
\end{subfigure}
\hspace{20mm}
\begin{subfigure}{0.3\textwidth}
\centering
\includegraphics[width=\textwidth]{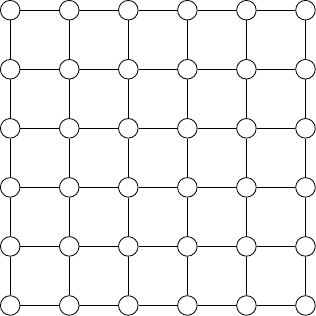}
\end{subfigure}

\caption{Left: Pariser--Parr--Pople model of benzene C$_6$H$_6$. Right: supercell model $(\Z/6\Z)^2$.\label{fig:Hubbard graphs}}
\end{figure}

\subsection{Matsubara frequency (MaF) discretization}

As explained above, the Matsubara frequency (MaF) discretization of IPT-DMFT equations aims at seeking approximations $(\HybridizationVector,\SelfEnergyVector) \in \LowerHalfPlaneVectors \times  \LowerHalfPlaneVectors$ of the values of the hybridization function $\Hybridization$ and local self-energy $\SelfEnergy$ at the lowest $(\MatsubaraFrequencyCutoff+1)$ imaginary Matsubara frequencies:
\begin{align}
\HybridizationVector = \left(\Hybridization_n\right)_{n \in \IntSubSet{0}{\MatsubaraFrequencyCutoff}}, \quad \Hybridization_n \simeq \Hybridization(i\MatsubaraFrequency_n),   \qquad
\SelfEnergyVector = \left(\SelfEnergy_n\right)_{n \in \IntSubSet{0}{\MatsubaraFrequencyCutoff}}, \quad  \SelfEnergy_n \simeq \SelfEnergy(i\MatsubaraFrequency_n),
\end{align}
for some Matsubara frequency cutoff $\MatsubaraFrequencyCutoff \in \Integers$. To formulate the MaF-discretized IPT-DMFT equations, it is convenient to set
\begin{equation}
\Hybridization_{-(n+1)}=\overline{\Hybridization_{n}} \in \overline{\UpperHalfPlane}, \quad \SelfEnergy_{-(n+1)}=\overline{\SelfEnergy_{n}} \in \overline{\UpperHalfPlane} \label{eq:ReflectionHybridizationSElfEnergy}
\end{equation}
in accordance with the conventional extension of a Pick function to the lower half-plane \cite{gesztesy_matrixvalued_2000} 
\begin{equation}\label{eq:extension_Pick_C-}
\forall z \in \LowerHalfPlane:=-\UpperHalfPlane, \quad f(z):=\overline{f(\overline{z})},
\end{equation}
and the fact that $\MatsubaraFrequency_{-(n+1)}=-\MatsubaraFrequency_n$. 

\medskip

The choice of the collocation points $\left(i\MatsubaraFrequency_n\right)_{n \in \IntSubSet{0}{\MatsubaraFrequencyCutoff}}$ in $\UpperHalfPlane$, represented in Figure \ref{fig:MatsubarafFrequencies}, is motivated by two reasons.
The first one is very generic. The imaginary Matsubara frequencies are the poles of the Fermi--Dirac function at inverse temperature $\StatisticalTemperature$ and therefore are natural collocation points when evaluating physical observables at Gibbs equilibrium using complex-plane contour integral techniques. 
The second reason is related to the expression of the IPT solver. The latter is formally defined at the imaginary Matsubara frequencies by the formula 
\begin{equation}\label{eq:Sigma_IPT_MF}
\forall n \in \N, \quad \IPTmap_\StatisticalTemperature(\Hybridization)(i\omega_n)=\int_{0}^\StatisticalTemperature e^{i\MatsubaraFrequency_n\tau} \left(\frac{1}{\StatisticalTemperature}\sum_{n' \in \RelativeIntegers }e^{-i\MatsubaraFrequency_{n'}\tau} \left(i\MatsubaraFrequency_{n'} - \Hybridization(i\MatsubaraFrequency_{n'})\right)^{-1}\right)^3 d\tau,
\end{equation}
where $\Hybridization$ is extended to negative imaginary Matsubara frequencies according to~\eqref{eq:extension_Pick_C-}.
It is well-know that an analytic function defined on an open subset $\Omega$ of the complex plane is fully characterized by its values on a countable number of points of $\Omega$ with an accumulation point in $\Omega$. But as the sequence of imaginary Matsubara frequencies $\left(i\MatsubaraFrequency_n\right)_{n \in \Integers}$ has no accumulation point in $\UpperHalfPlane$, it is not clear that~\eqref{eq:Sigma_IPT_MF} fully defines $\IPTmap_\StatisticalTemperature$. It is however the case if the function $\Hybridization$ is the negative of a meromorphic Pick function with a finite number of poles. It can then be shown that the so-defined map can be uniquely extended by (weak-)continuity to a function $\IPTmap_\StatisticalTemperature$ from the set of physically-admissible hybridization function to the set of physically-admissible local self-energies. We refer to ~\cite{cances_mathematical_2024} for more details on this point.

In view of~\eqref{eq:Sigma_IPT_MF}, a natural MaF discretization of~\eqref{eq:IPTUpperHalfPlane} is 
$$
\SelfEnergyVector =U^2 \IPTmap_{\StatisticalTemperature,\MatsubaraFrequencyCutoff}(\HybridizationVector),
$$
with 
\begin{equation}\label{eq:def_FNomegabeta}
\forall n \in \IntSubSet{0}{\MatsubaraFrequencyCutoff}, \quad [\IPTmap_{\StatisticalTemperature,\MatsubaraFrequencyCutoff}(\HybridizationVector)]_n :=\int_{0}^\StatisticalTemperature e^{i\MatsubaraFrequency_n \tau} \left(\frac{1}{\StatisticalTemperature} \sum_{n'=-(\MatsubaraFrequencyCutoff+1)}^{\MatsubaraFrequencyCutoff} e^{-i\MatsubaraFrequency_{n'} \tau}\left(i\MatsubaraFrequency_{n'}-\Hybridization_{n'} \right)^{-1}\right)^3 d\tau.
\end{equation}
Note that the discretization of \eqref{eq:IPTUpperHalfPlane} is nonlocal in the sense that $\SelfEnergy_n$ depends on $\left(\Hybridization_{n'}\right)_{n' \in \IntSubSet{0}{\MatsubaraFrequencyCutoff}}$ and not only on $\Hybridization_n$.

\begin{figure}[h]
\centering
\begin{tikzpicture}
%Drawing axis
\draw[->] (-1cm,0) -- (2cm,0) node[right] {$\Re(z)$};
\draw[->] (0,-1cm) -- (0,5cm) node[above] {$\Im(z)$};
%Placing chemical potential on the real line.
%\draw[-] (2cm,2pt) -- (2cm,-2pt) node[below] {$0$};
%Placing Matsubara frequencies on the imaginary line.
\path (0,0.5cm)
 foreach \y/\ytext in {0/, 1/3,2/5,3/7,4/9} {
    pic [yshift=\y cm] {code={\draw (2pt,0) node[left=0.25cm] {$\ytext \pi/\StatisticalTemperature$}--(-2pt,0);}}
  }
  -- (0,4.5cm);
%Placing the points of the analytical continuation problem.
\path (0cm,0.5cm) 
foreach \y in {0,1,2,3,4} {
	pic [yshift=\y cm] {code={\node [right=0.25 cm] {$i\MatsubaraFrequency_{\y}$}--(-2pt,-2pt);}}
  }
  -- (2cm,4.5cm);

\end{tikzpicture}
\caption{The set $\left(i\MatsubaraFrequency_n\right)_{n \in \IntSubSet{0}{\MatsubaraFrequencyCutoff}}$ of points in $\UpperHalfPlane$ used for the discretization of $\Hybridization$ and $\SelfEnergy$ with $\MatsubaraFrequencyCutoff=4$}
\label{fig:MatsubarafFrequencies}
\end{figure}

\medskip

 We consolidate the content of this section into the following definition.

\begin{definition}[MaF-discretized IPT-DMFT equations] \label{def:MatsubarasEquations}
Consider an undirected, loopless, vertex transitive graph $\HubbardGraph=(\HubbardVertices,\HubbardEdges)$ with $\Cardinal{\HubbardVertices}=L \in \Integers$ and associated adjacency matrix $\AdjacencyMatrix \in \RealNumbers^{L \times L}_{\rm sym}$.
Given a hopping amplitude $\HoppingMatrix[] \in \RealNumbers$, introduce $\NIHamiltonian_\perp \in \RealNumbers^{(L-1) \times (L-1)}_{\rm sym}$ and $W \in \RealNumbers^{L-1}$ such that
\begin{equation}
\HoppingMatrix[] \,  \AdjacencyMatrix= \left(\begin{matrix}
0 & W^T \\
W & \NIHamiltonian_\perp
\end{matrix}\right). \nonumber
\end{equation}
Given an on-site interaction energy $\OnSiteRepulsion[] \in \RealNumbers$, an inverse temperature $\StatisticalTemperature \in \RealNumbers_{+}^*$ and a Matsubara frequencies cutoff $\MatsubaraFrequencyCutoff \in \Integers$, the \emph{MaF-discretized \acrshort{ipt}-\acrshort{dmft} equations} consist of the following set of equations: find $\HybridizationVector =\left(\Hybridization_n \right)_{n \in \IntSubSet{0}{\MatsubaraFrequencyCutoff}} \in \LowerHalfPlaneVectors$ and $\SelfEnergyVector=\left(\SelfEnergy_n\right)_{n \in \IntSubSet{0}{\MatsubaraFrequencyCutoff}} \in \LowerHalfPlaneVectors$ such that 
\begin{align}
\forall n \in \IntSubSet{0}{\MatsubaraFrequencyCutoff}, \quad  \Hybridization_n&=W^T\left(i\MatsubaraFrequency_n - \NIHamiltonian_\perp - \SelfEnergy_n \right)^{-1}W, \label{eq:BUMFDiscretized}\\
\SelfEnergyVector &=\OnSiteRepulsion[]^2 \IPTmap_{\StatisticalTemperature,\MatsubaraFrequencyCutoff}(\HybridizationVector), \label{eq:IPTMFDiscretized}
\end{align}
where $\MatsubaraFrequency_n=\frac{2(n+1)\pi}{\StatisticalTemperature}$ is the $n$-th Matsubara frequency and $\IPTmap_{\StatisticalTemperature,\MatsubaraFrequencyCutoff}$ is defined by \eqref{eq:def_FNomegabeta}.
\end{definition}

\begin{remark}[comment on the set of admissible solutions]\label{rmk:AdmissibleSolutions}
    Note that we restrict ourselves to solutions satisfying the condition $\HybridizationVector,\SelfEnergyVector \in \LowerHalfPlaneVectors$.
    This is motivated by the fact that $\Hybridization$ and $\SelfEnergy$ are negatives of Pick functions.
    Nevertheless, for $\MatsubaraFrequencyCutoff > 0$, not every element of $\LowerHalfPlaneVectors$ can be seen as the set of values of some negative of a Pick function at the lowest $(\MatsubaraFrequencyCutoff +1)$ imaginary Matsubara frequencies. The problem of finding a Pick function that interpolates these values is known as the Nevanlinna--Pick interpolation problem, and the reader is referred to \cite{cances_mathematical_2024} for a brief overview of this topic.
    In particular, nothing ensures \emph{a priori} that the values $\boldsymbol{ \Hybridization}$ or $\boldsymbol{\SelfEnergy}$ obtained by solving~\eqref{eq:BUMFDiscretized}--\eqref{eq:IPTMFDiscretized} can be interpolated by the negative of a Pick function. The numerical simulations reported in Section \ref{sec:IterativeSchemeMottTransiton} reveal that, for a large set of parameters, this is indeed not the case. This issue is well-identified when using approximate impurity solvers in \acrshort{dmft} computations, as already pointed out in e.g.~\cite{fei_nevanlinna_2021,fei_analytical_2021}.
\end{remark}

\subsection{Dimensionless formulation}

Let us now provide a dimensionless version of the MaF-discretized IPT-DMFT equations suitable for mathematical analysis.
To do so, consider the following dimensionless quantities (recall that both $\Hybridization$ and $\SelfEnergy$ have energy dimensions, while $\StatisticalTemperature$ has the dimension of the inverse of an energy)
\begin{equation}
   \HybridizationVector'= \StatisticalTemperature\HybridizationVector, \quad \SelfEnergyVector'=\StatisticalTemperature\SelfEnergyVector.
\end{equation}
Obviously, $(\HybridizationVector,\SelfEnergyVector)$ satisfies equations \eqref{eq:BUMFDiscretized}--\eqref{eq:IPTMFDiscretized} if and only if $(\HybridizationVector',\SelfEnergyVector')$ satisfies the following equations
\begin{align}
\forall  n \in\IntSubSet{0}{\MatsubaraFrequencyCutoff}, \qquad   \Hybridization'_n &=t^2 w^T\left(i(2n+1)\pi - t h^0_\perp - \SelfEnergy'_n\right)^{-1}w, \\
    \SelfEnergy'_n &= u^2\int_0^1 e^{i(2n+1)\pi \tau'} \left(\sum_{n'=-(\MatsubaraFrequencyCutoff+1)}^{\MatsubaraFrequencyCutoff} e^{-i(2n'+1) \pi \tau'}\left(i(2n'+1)\pi-\Hybridization'_{n'} \right)^{-1}\right)^3 d\tau',\label{eq:AlmostUndimmensionnalizedSelfEnergy}
\end{align}
with 
$$
t=\StatisticalTemperature\HoppingMatrix[], \qquad u=\StatisticalTemperature \OnSiteRepulsion[], \quad w \in \{0,1\}^{L-1}, \; h^0_\perp \in \{0,1\}_{(L-1)\times(L-1)} \mbox{ such that }  \AdjacencyMatrix=\left(\begin{matrix}
        0 & w^T \\
        w & h^0_\perp
    \end{matrix}\right).
$$
Since for all $n \in \IntSubSet{0}{\MatsubaraFrequencyCutoff}$,
\begin{align}
    &\int_0^1 e^{i(2n+1)\pi \tau'} \left(\sum_{n'=-(\MatsubaraFrequencyCutoff+1)}^{\MatsubaraFrequencyCutoff} e^{-i(2n'+1) \pi \tau'}\left(i(2n'+1)\pi-\Hybridization'_{n'} \right)^{-1}\right)^3 d \tau' = \\
    &\sum_{n_1,n_2,n_3=-(\MatsubaraFrequencyCutoff+1)}^{\MatsubaraFrequencyCutoff}\prod_{j=1}^3 \left(i(2n_j+1)\pi - \Hybridization'_{n_j}\right)^{-1} \int_{0}^1 e^{i2\pi\left(n-1-n_1-n_2-n_3\right)\tau'}d\tau'= \nonumber \\
    &\sum_{\substack{n_1,n_2,n_3=-(\MatsubaraFrequencyCutoff+1) \\ n_1 + n_2 + n_3 =n-1}}^{\MatsubaraFrequencyCutoff} \prod_{j=1}^3\left(i(2n_j+1)\pi - \Hybridization'_{n_j}\right)^{-1},
\end{align}
\eqref{eq:AlmostUndimmensionnalizedSelfEnergy} can be rewritten as 
\begin{equation}\label{eq:NondimensionalIPTOutOfDefinition}
   \bm  \SelfEnergy'=u^2 F_{\MatsubaraFrequencyCutoff}(\HybridizationVector'),
\end{equation}
where  $F_{\MatsubaraFrequencyCutoff}: \LowerHalfPlaneVectors \to \ComplexNumbers^{\MatsubaraFrequencyCutoff+1}$ is the rational fraction defined by, for all $\mathbf{z}\in \UpperHalfPlane^{\MatsubaraFrequencyCutoff+1}$
\begin{equation}
F_{\MatsubaraFrequencyCutoff}(\bm z)_{n}=\sum_{\substack{n_1,n_2,n_3=-(\MatsubaraFrequencyCutoff+1) \\ n_1 + n_2 + n_3 =n-1}}^{\MatsubaraFrequencyCutoff} \prod_{j=1}^3\left(i(2n_j+1)\pi - z_{n_j} \right)^{-1},\label{eq:FunctionCriticalRadius}
\end{equation}
where for all $n \in \IntSubSet{-\MatsubaraFrequencyCutoff+1}{-1}, z_{n}:=\overline{z_{-(n+1)}}$ (as in \eqref{eq:ReflectionHybridizationSElfEnergy}).

\medskip

Replacing $\Hybridization',\SelfEnergy'$ by $\Hybridization,\SelfEnergy$ in order to simplify the notation, we end up with the following dimensionless formulation of the MaF-discretization of the  \acrshort{ipt}-\acrshort{dmft} equations.
\begin{definition}[dimensionless MaF-discretized \acrshort{ipt}-\acrshort{dmft} equations]
\label{def:NondimensionalEquations}
Consider an undirected, loopless, vertex transitive graph $\HubbardGraph=(\HubbardVertices,\HubbardEdges)$ with $\Cardinal{\HubbardVertices}=L \in \Integers^*$ vertices and adjacency matrix $\AdjacencyMatrix$, a hopping parameter  $\HoppingMatrix[] \in \RealNumbers$, an on-site interaction energy $U \in \R$, and an inverse temperature $\StatisticalTemperature \in \RealNumbers_{+}^*$. The dimensionless MaF-discretized IPT-DMFT equations reads as follows: find $\HybridizationVector = (\Hybridization_n)_{n \in\IntSubSet{0}{\MatsubaraFrequencyCutoff}} \in \LowerHalfPlaneVectors$ and $\SelfEnergyVector = (\SelfEnergy_n)_{n \in\IntSubSet{0}{\MatsubaraFrequencyCutoff}}  \in \LowerHalfPlaneVectors$ satisfying
    \begin{align}
   \forall   n \in \IntSubSet{0}{\MatsubaraFrequencyCutoff}, \qquad   \Hybridization_n &=t^2 w^T\left(i(2n+1)\pi - t h^0_\perp - \SelfEnergy_n\right)^{-1}w \label{eq:NondimensionalBU}\\
        \SelfEnergy_n&=u^2 F_{\MatsubaraFrequencyCutoff}(\HybridizationVector)_n, \label{eq:NondimensionalIPT}
    \end{align}
    with $t=\StatisticalTemperature \HoppingMatrix[]$, $u=\StatisticalTemperature \OnSiteRepulsion[]$, $F_{\MatsubaraFrequencyCutoff}$ given by \eqref{eq:FunctionCriticalRadius} and $w \in \R^{L-1}$ and $h^0_\perp \in \R^{(L-1) \times (L-1)}_{\rm sym}$ such that
    \begin{equation}
        \AdjacencyMatrix=\left( \begin{matrix}
            0 & w^T \\
            w & h^0_\perp
        \end{matrix} \right).
    \end{equation}
\end{definition}

\begin{remark}[comparison with other dimensionless versions of the IPT-DMFT equations]

In the existing literature \cite{georges_dynamical_1996}, the dimensionless quantities are rather expressed in units of the hopping parameter~$\HoppingMatrix[]$, or, more precisely, in units of the half-band width 
\begin{equation}
    D:=\HoppingMatrix[] \, \mathrm{diam}(\sigma(\AdjacencyMatrix))/2,
\end{equation}
where $\mathrm{diam}(\sigma(\AdjacencyMatrix))$ is the diameter of the spectrum $\sigma(\AdjacencyMatrix)$ of the real-symmetric adjacency matrix.
This choice is made to facilitate the comparison with experimental data, such as temperature-pressure phase diagram of transition metal oxides \cite{georges_dynamical_1996}, but is less suited to the mathematical analysis of the equations.
\end{remark}

\section{Existence and uniqueness results}\label{sec:MainResults}

In this section, we present our theoretical results regarding the solution(s) to the dimensionless MaF-discretized IPT-DMFT equations \eqref{eq:NondimensionalBU}--\eqref{eq:NondimensionalIPT}.

\subsection{Existence of solutions}

Using Brouwer's fixed-point theorem \cite{shapiro_fixed-point_2016}, we are able to show that \eqref{eq:BUMFDiscretized}--\eqref{eq:IPTMFDiscretized} admit a solution, but only within a specific parameter range, except for $\MatsubaraFrequencyCutoff =0$. This limitation arises from the fact that, in general and for $\MatsubaraFrequencyCutoff > 0$, the map $F_{\MatsubaraFrequencyCutoff} : \C^{\MatsubaraFrequencyCutoff+1} \to \C^{\MatsubaraFrequencyCutoff+1}$ does not map $\LowerHalfPlaneVectors$ onto itself. For instance, for $\MatsubaraFrequencyCutoff=1$ and $\HybridizationVector=(-i\pi x_0, -3i\pi x_1) \in -\overline{\UpperHalfPlane}^2$ with  $x_0,x_1  \in \RealNumbers_+$, we have
\begin{equation} \label{eq:CounterExampleF}
    F_{1}(\HybridizationVector)_1=-\frac{i}{\left((1+x_0)\pi\right)^3}\left(\frac{1}{9} \left(\frac{1+x_0}{1+x_1}\right)^{3} + 2 \frac{1+x_0}{1+x_1} -1 \right) \xrightarrow{x_1 \to +\infty} \frac{i}{\left((1+x_0)\pi\right)^3} \notin -\overline{\UpperHalfPlane}.
\end{equation}
Such a problem does not hold for $\MatsubaraFrequencyCutoff =0$, for which we have
  \begin{equation} \label{eq:F0}
 \forall \Hybridization \in -\UpperHalfPlane, \quad F_0(\Hybridization)=\frac{3}{\vert i\pi-\Hybridization|^2(i\pi-\Hybridization)} \in -\overline{\UpperHalfPlane}.
    \end{equation}

\begin{theorem}[Existence of solution to the discretized \acrshort{ipt}-\acrshort{dmft} equations.]\label{thm:GlobalExistenceMFDiscretized}
Let $\MatsubaraFrequencyCutoff \in \Integers^*$ and 
\begin{equation}
\label{eq:StableSpace}
C_{\MatsubaraFrequencyCutoff} := \left\{ \bm z=(z_n)_{n \in \IntSubSet{0}{\MatsubaraFrequencyCutoff}}  \in \LowerHalfPlaneVectors \; \big| \; |z_n| \le \frac{1}{(2n+1)\pi} \right\}.
\end{equation}
There exists $t_{\MatsubaraFrequencyCutoff} > 0$ depending only on $\MatsubaraFrequencyCutoff$ such that for all
$0 \le t \le \frac{t_{\MatsubaraFrequencyCutoff}}{\sqrt{\mathrm{deg}(\HubbardGraph)}}$, the MaF-discretized \acrshort{ipt}-\acrshort{dmft} equations~\eqref{eq:NondimensionalBU}--\eqref{eq:NondimensionalIPT} have a solution 
\begin{equation}\label{eq:bounds_on_the_solution}
(\HybridizationVector,\SelfEnergyVector) \in \HybridizationSpaceDiscretized  \times u^2 F_{\MatsubaraFrequencyCutoff}\left(\HybridizationSpaceDiscretized\right) \subset \LowerHalfPlaneVectors \times \LowerHalfPlaneVectors,
\end{equation}
where $\HybridizationSpaceDiscretized$ is given by
\begin{equation}
    \label{eq:HybridizationSpaceDiscretized}
    \HybridizationSpaceDiscretized=t^2  \deg(\HubbardGraph) C_{\MatsubaraFrequencyCutoff}=\left\{ \bm z=(z_n)_{n \in \IntSubSet{0}{\MatsubaraFrequencyCutoff}}  \in \LowerHalfPlaneVectors \; \big| \; |z_n| \le \frac{t^2\deg(\HubbardGraph)}{(2n+1)\pi} \right\}.
\end{equation}
\end{theorem}

\medskip
\begin{remark}[on the non-interacting setting $u=0$]
    As discussed in the next section, the MaF-discretized IPT-DMFT equations admit a unique solution for $u=0$, characterized by
    \begin{equation}
        \forall n \in \IntSubSet{0}{\MatsubaraFrequencyCutoff}, \quad \SelfEnergy_n=0,  \quad \Hybridization_n= t^2w^T\left(i(2n+1)\pi-t h^0_\perp \right)w,
    \end{equation}
    which satisfies the condition $\HybridizationVector,\SelfEnergyVector \in \LowerHalfPlaneVectors$, independently of the value of $t$. This existence and uniqueness result does not extend automatically to the small $u$ regime by perturbation, as $-\SelfEnergyVector$ is on the boundary of the admissible domain $\LowerHalfPlaneVectors$ for $u=0$.
\end{remark}\medskip

\begin{proof}[Proof of Theorem~\ref{thm:GlobalExistenceMFDiscretized}]
The proof is based on Brouwer's fixed point theorem. Let $\BathUpdateDiscretized$ be defined, for all  $n \in \IntSubSet{0}{\MatsubaraFrequencyCutoff}$, by
\begin{equation}
\BathUpdateDiscretized(\SelfEnergyVector)_n =t^2 w^T\left(i(2n+1)\pi - th^0_\perp - \SelfEnergy_n\right)^{-1}w, \nonumber
\end{equation} so that \eqref{eq:NondimensionalBU} $\iff \HybridizationVector=\BathUpdateDiscretized(\SelfEnergyVector)$.
We start with the following lemma.

\begin{lemma}\label{lemma:BoundednessBUMFDiscretized}
For $\SelfEnergyVector\in\LowerHalfPlaneVectors$, $\BathUpdateDiscretized(\SelfEnergyVector)$ is well defined and belongs to $\LowerHalfPlaneVectors$. Moreover, for $n\in\IntSubSet{0}{\MatsubaraFrequencyCutoff}$,
\begin{equation}\nonumber
\Modulus{\BathUpdateDiscretized(\SelfEnergyVector)_n} \leq \frac{t^2\deg(\HubbardGraph)}{(2n+1)\pi} . 
\end{equation}
In other words, $\BathUpdateDiscretized(\SelfEnergyVector)\in \HybridizationSpaceDiscretized $, where $\HybridizationSpaceDiscretized$ is defined in equation~\eqref{eq:HybridizationSpaceDiscretized}.
\end{lemma}

\begin{proof}
The sign of the imaginary part is a direct consequence of the definition of the imaginary part of matrices and of the fact that if $f$ is a nonzero Pick matrix, that is $(f-f^\dagger)/(2i)$ is definite positive, then so is $-f^{-1}$. Recall that $h^0_\perp \in \SelfAdjointOperator_{L-1}(\RealNumbers)$, so that there exists $P \in \mathcal{M}_{L-1}(\ComplexNumbers)$ unitary with $h^0_\perp=P\diag(\eps_1,\ldots,\eps_{L-1})P^\dagger$. From that, we have for all $n \in \IntSubSet{0}{\MatsubaraFrequencyCutoff}$,
\begin{align}
t^{-2}\Modulus{\BathUpdateDiscretized(\SelfEnergyVector)_n} & \leq \sum_{k=1}^{L-1} \frac{\Modulus{(wP)_k}^2}{\Modulus{i(2n+1)\pi - \SelfEnergy_n - t\eps_k}} \leq \frac{\Norm{w}_2^2}{\left|(2n+1)\pi-\Im(\SelfEnergy_n) \right|}  \leq \frac{\mathrm{deg}(\HubbardGraph)}{(2n+1)\pi}, \label{eq:BUComponentWiseEstimate}
\end{align}
where we have used that $\Norm{w}_2^2=\mathrm{deg}(\HubbardGraph)$, because $\HubbardGraph$ is vertex transitive.
\end{proof}

The analogue of this lemma for $F_{\MatsubaraFrequencyCutoff}$ is more restrictive (see the counterexample given in equation \eqref{eq:CounterExampleF}), as we state now.

\begin{lemma}\label{lemma:IPTDiscretizedFormula}

For all Matsubara frequency cutoff $\MatsubaraFrequencyCutoff \in \N^\ast$, there exists $t_{\MatsubaraFrequencyCutoff}>0$ such that for all $0 \le t\leq t_{\MatsubaraFrequencyCutoff}$, $F_{\MatsubaraFrequencyCutoff}(\HybridizationSpaceDiscretized)\subset \LowerHalfPlaneVectors$.
\end{lemma}

\begin{proof}
Let us first prove that $\Im(F_{\MatsubaraFrequencyCutoff}(0)_n)< 0$. For all $n \in \IntSubSet{0}{\MatsubaraFrequencyCutoff}$, we define
\begin{equation} \nonumber
\varphi(n,\MatsubaraFrequencyCutoff) = \pi^3 \Im(F_{\MatsubaraFrequencyCutoff}(0)_n)=\sum_{\substack{n_1,n_2,n_3=-(\MatsubaraFrequencyCutoff+1) \\ n_1 + n_2 + n_3 =n-1}}^{\MatsubaraFrequencyCutoff} \prod_{i=1}^3\frac{1}{2n_i+1}
\end{equation}
Let us show that the positive rational number $\varphi(n,\MatsubaraFrequencyCutoff)$ increases with $\MatsubaraFrequencyCutoff$: by enumerating, we have for all $n \in \IntSubSet{0}{\MatsubaraFrequencyCutoff}$,
\begin{align}
\varphi(n,\MatsubaraFrequencyCutoff+1)-\varphi(n,\MatsubaraFrequencyCutoff)&=\frac{3}{2\MatsubaraFrequencyCutoff+3}\sum_{\substack{n_1,n_2=-(\MatsubaraFrequencyCutoff+1) \\ n_1 + n_2 + \MatsubaraFrequencyCutoff+1 = n-1}}^{\MatsubaraFrequencyCutoff} \frac{1}{(2n_1+1)(2n_2+1)}  \nonumber\\
&- \frac{3}{2\MatsubaraFrequencyCutoff+3} \sum_{\substack{n_1,n_2=-(\MatsubaraFrequencyCutoff+1) \\ n_1 + n_2 - (\MatsubaraFrequencyCutoff+2) = n-1}}^{\MatsubaraFrequencyCutoff} \frac{1}{(2n_1+1)(2n_2+1)}  \nonumber \\
&- \frac{6}{(2\MatsubaraFrequencyCutoff+3)^2(2n+1)}. \label{eq:phiNomega+1-phiNomega}
\end{align}
Now for the first term in the \acrshort{rhs} of \eqref{eq:phiNomega+1-phiNomega}, we have for all $n_1,n_2 \in \IntSubSet{-(\MatsubaraFrequencyCutoff+1)}{\MatsubaraFrequencyCutoff}$ satisfying $n_1+n_2+\MatsubaraFrequencyCutoff+1=n-1$,
\begin{equation}\nonumber
\frac{1}{(2n_1+1)(2n_2+1)}=\frac{1}{2(n-(\MatsubaraFrequencyCutoff+1))}\left(\frac{1}{2n_1+1}+\frac{1}{2n_2+1}\right)
\end{equation}
so that 
\begin{align}
\sum_{\substack{n_1,n_2=-(\MatsubaraFrequencyCutoff+1) \\ n_1 + n_2 + \MatsubaraFrequencyCutoff+1 = n-1}}^{\MatsubaraFrequencyCutoff} \frac{1}{(2n_1+1)(2n_2+1)} &= \frac{1}{n-(\MatsubaraFrequencyCutoff+1)}\sum_{n_1=-(\MatsubaraFrequencyCutoff+1)}^{n-1} \frac{1}{2n_1+1}\nonumber \\
&=\frac{1}{\MatsubaraFrequencyCutoff+1-n}\sum_{n_1=n}^{\MatsubaraFrequencyCutoff}\frac{1}{2n_1+1},\nonumber
\end{align}
where the last equality comes from the fact that $\sum_{n_1=-(\MatsubaraFrequencyCutoff+1)}^{\MatsubaraFrequencyCutoff}\frac{1}{2n_1+1}=0$.
For the second term, we also have for all $n_1,n_2 \in \IntSubSet{-(\MatsubaraFrequencyCutoff+1)}{\MatsubaraFrequencyCutoff}$ satisfying $n_1+n_2-(\MatsubaraFrequencyCutoff+2)=n-1$,
\begin{equation}\nonumber
\frac{1}{(2n_1+1)(2n_2+1)}=\frac{1}{2(n+\MatsubaraFrequencyCutoff+2)}\left(\frac{1}{2n_1+1}+\frac{1}{2n_2+1}\right),
\end{equation}
so that 
\begin{equation}\nonumber
\sum_{\substack{n_1,n_2=-(\MatsubaraFrequencyCutoff+1) \\ n_1 + n_2 - (\MatsubaraFrequencyCutoff+2) = n-1}}^{\MatsubaraFrequencyCutoff} \frac{1}{(2n_1+1)(2n_2+1)}= \frac{1}{n+\MatsubaraFrequencyCutoff+2}\sum_{n_1=n+1}^{\MatsubaraFrequencyCutoff}\frac{1}{2n_1+1}.
\end{equation}
We end up with 
\begin{align}
\varphi(n,\MatsubaraFrequencyCutoff+1)-\varphi(n,\MatsubaraFrequencyCutoff)
=&\frac{3}{2\MatsubaraFrequencyCutoff+3}\left( \frac{1}{\MatsubaraFrequencyCutoff+1-n}-\frac{1}{n+\MatsubaraFrequencyCutoff+2}\right)\sum_{n_1=n+1}^{\MatsubaraFrequencyCutoff} \frac{1}{2n_1+1} \nonumber \\
&+\frac{3}{(2\MatsubaraFrequencyCutoff+3)(2n+1)}\left(\frac{1}{(\MatsubaraFrequencyCutoff+1-n)}-\frac{2}{2\MatsubaraFrequencyCutoff+3}\right), \nonumber
\end{align}
which is positive, hence $\varphi(n,\MatsubaraFrequencyCutoff)$ increases with $\MatsubaraFrequencyCutoff$ and so is $\Im(F_{\MatsubaraFrequencyCutoff}(0)_n)$. From equations \eqref{eq:Sigma_IPT_MF} and \eqref{eq:def_FNomegabeta}, we obtain, using \eqref{eq:Sigma_IPT_MF}, \eqref{eq:def_FNomegabeta}, \cite[Proposition 2.18]{cances_mathematical_2024}, and the dominated convergence theorem that
$$
\lim_{\MatsubaraFrequencyCutoff \to \infty} \Im\left( F_{\MatsubaraFrequencyCutoff}(0)_n \right) = \frac{1}{\StatisticalTemperature} \Im\left(\IPTmap_\StatisticalTemperature(0)(i\MatsubaraFrequency_n)\right) < 0.
$$
Therefore,
\begin{equation}
\Im(F_{\MatsubaraFrequencyCutoff}(0)_n) \leq \frac{1}{\StatisticalTemperature}\Im(\IPTmap_{\StatisticalTemperature}(0)(i\MatsubaraFrequency_n)) < 0. \nonumber
\end{equation}
This shows that $F_{\MatsubaraFrequencyCutoff}(0)\in \LowerHalfPlaneVectors$. This set is open and $F_{\MatsubaraFrequencyCutoff}$ is continuous on $\LowerHalfPlaneVectors$. Since $\HybridizationSpaceDiscretized$ is a compact convex set containing $z=0$, we conclude by defining $\alpha_{\MatsubaraFrequencyCutoff}=\sup\left\{\alpha>0: F_{\MatsubaraFrequencyCutoff}(\alpha C_{\MatsubaraFrequencyCutoff})\subset\LowerHalfPlaneVectors\right\}$, and $t_{\MatsubaraFrequencyCutoff}=\sqrt{\alpha_{\MatsubaraFrequencyCutoff}}$.
\end{proof}

 We know that $0<t_{\MatsubaraFrequencyCutoff}$, and for $\MatsubaraFrequencyCutoff\geq 1$, $t_{\MatsubaraFrequencyCutoff}<\infty$, as shown in the counter-example~\eqref{eq:CounterExampleF}.
We conclude from the two previous lemmata, that for $t<t_{\MatsubaraFrequencyCutoff}/\sqrt{\deg(\HubbardGraph)}$, the map $\DMFTmapDiscretized: \HybridizationVector \mapsto \BathUpdateDiscretized(u^2F_{\MatsubaraFrequencyCutoff}(\HybridizationVector))$ is well-defined and maps $\HybridizationSpaceDiscretized$ onto itself: indeed, for all $\HybridizationVector \in \HybridizationSpaceDiscretized$, $u^2F_{\MatsubaraFrequencyCutoff}(\HybridizationVector) \in \LowerHalfPlaneVectors$ by lemma~\ref{lemma:IPTDiscretizedFormula}, hence $\BathUpdateDiscretized(u^2F_{\MatsubaraFrequencyCutoff}(\HybridizationVector)) \in \HybridizationSpaceDiscretized\subset\LowerHalfPlaneVectors$ by Lemma~\ref{lemma:BoundednessBUMFDiscretized}.

Since $\DMFTmapDiscretized$ is continuous and $\HybridizationSpaceDiscretized$ is compact, $\DMFTmapDiscretized$ has a fixed-point in $\HybridizationSpaceDiscretized$ by Brouwer's fixed-point theorem, which concludes the proof.
\end{proof}

\subsection{Uniqueness} \label{sec:LocalUniquenessIPTDMFT}

In this section, we state a result which guarantees the uniqueness of the solution. Note first that, similarly as for the continuous \acrshort{ipt}-\acrshort{dmft} equations \eqref{eq:BathUpdateUpperHalfPlane}--\eqref{eq:IPTUpperHalfPlane} (see \cite[Proposition 2.12]{cances_mathematical_2024}), the MaF-discretized IPT-DMFT equations \eqref{eq:NondimensionalBU}--\eqref{eq:NondimensionalIPT} admit a unique solution in the following trivial limits:
\begin{itemize}
\item for $u=0$, the unique solution is given by
\begin{equation} \nonumber
\forall n \in \IntSubSet{0}{\MatsubaraFrequencyCutoff}, \quad\SelfEnergy_n=0,\quad \Hybridization_n= t^2w^T\left(i(2n+1)\pi-t h^0_\perp \right)w;
\end{equation}
\item for $t=0$, the unique solution is given by
\begin{equation} \nonumber
\forall n \in \IntSubSet{0}{\MatsubaraFrequencyCutoff}, \quad \SelfEnergy_n=u^2F_{\MatsubaraFrequencyCutoff}(0)_n,\quad \Hybridization_n=0.
\end{equation}
\end{itemize}
In both cases, these solutions satisfy the condition $\HybridizationVector,\SelfEnergyVector \in \LowerHalfPlaneVectors$ (see the proof of Lemma \ref{lemma:IPTDiscretizedFormula} for the second case). The following results cover the perturbative regime around these two limiting cases: it proves the existence and uniqueness of a solution satisfying the bounds~\eqref{eq:bounds_on_the_solution}, as well as the linear convergence of the discretized IPT-DMFT simple fixed-point iteration scheme observed in numerical simulations in both perturbation regimes (see Section \ref{sec:IterativeSchemeMottTransiton}). 

%As it holds for Theorem \ref{thm:GlobalExistenceMFDiscretized}, our analysis relies on the properties of $F_{\MatsubaraFrequencyCutoff}$: let us introduce $L_{\MatsubaraFrequencyCutoff}$ the (finite) largest Lipschitz constant of the maps $F_{\MatsubaraFrequencyCutoff,n}: z \mapsto F_{\MatsubaraFrequencyCutoff}(z)_n$ over $\LowerHalfPlaneVectors$, for $n \in \IntSubSet{0}{\MatsubaraFrequencyCutoff}$:
%\begin{equation}
%L_{\MatsubaraFrequencyCutoff}=\max_{n \in \IntSubSet{0}{\MatsubaraFrequencyCutoff}} \mathrm{Lip}_{\LowerHalfPlaneVectors}(F_{\MatsubaraFrequencyCutoff,n}).
%\end{equation}
%Note incidentally that, since $z \mapsto F_{\MatsubaraFrequencyCutoff}(z)_n$ is a sum of rational function in the $(\Hybridization_n)_{n \in \IntSubSet{0}{\MatsubaraFrequencyCutoff}}$ with poles on the positive imaginary axis $i\RealNumbers_{+}^*$, we have for all $R>0$, $\mathrm{Lip}_{\LowerHalfPlaneVectors}(F_{\MatsubaraFrequencyCutoff,n})=\mathrm{Lip}_{B(0,R)\cap(\LowerHalfPlaneVectors)}(F_{\MatsubaraFrequencyCutoff,n})$.

\begin{theorem}[Uniqueness of the solution to the discretized \acrshort{ipt}-\acrshort{dmft} equations]\label{thm:LocalUniquenessMFDiscretized}
Let $\MatsubaraFrequencyCutoff \in \Integers$. There exist constants $\eta_{\MatsubaraFrequencyCutoff} > 0$ and $t_{\MatsubaraFrequencyCutoff}>0$ depending only on $\MatsubaraFrequencyCutoff$, such that if the Hubbard graph $\HubbardGraph$  and the parameters $t,u \in \RealNumbers$ satisfy
\begin{equation}\label{eq:AssumptionExistenceTheorem}
0 \le t<t_{\MatsubaraFrequencyCutoff}/\sqrt{\deg(\HubbardGraph)},
\end{equation}
\begin{equation}\label{eq:AssumptionUniquenessTheorem}
    t^2 u^2 \mathrm{deg}(\HubbardGraph)  <  \eta_{\MatsubaraFrequencyCutoff},
\end{equation}
then the dimensionless MaF-discretized IPT-DMFT equations \eqref{eq:NondimensionalBU}--\eqref{eq:NondimensionalIPT} have a unique solution
$(\HybridizationVector_\star,\SelfEnergyVector_\star)$ in $\HybridizationSpaceDiscretized  \times  u^2 F_{\MatsubaraFrequencyCutoff}\left(\HybridizationSpaceDiscretized\right)$,
where $\HybridizationSpaceDiscretized$ is defined in~\eqref{eq:HybridizationSpaceDiscretized}.
Moreover, the sequence $\left(\HybridizationVector^{(k)}\right)_{k \in \Integers}$ generated by the simple fixed-point iteration scheme
\begin{equation}\label{eq:IterativeSchemeEquation}
\HybridizationVector^{(0)} \in \HybridizationSpaceDiscretized, \quad \forall k \in \Integers, \quad \HybridizationVector^{(k+1)}= \DMFTmapDiscretized(\HybridizationVector^{(k)})
\end{equation}
where $\DMFTmapDiscretized$ is defined by
\begin{equation}\label{eq:DMFTMapDefinition}
\forall n \in \IntSubSet{0}{\MatsubaraFrequencyCutoff}, \quad    [\DMFTmapDiscretized(\HybridizationVector)]_n=t^2 w^T \left(i2(n+1)\pi - t h^0_\perp -u^2[F_{\MatsubaraFrequencyCutoff}(\HybridizationVector)]_n\right)^{-1}w,
\end{equation}
is well-defined and converges linearly to $\HybridizationVector_\star$.
\end{theorem}

The proof is based on the Picard fixed-point theorem on $\HybridizationSpaceDiscretized$, for $t$ small enough. We endow here $\HybridizationSpaceDiscretized$ with the distance induced by the maximum norm $\Norm{\cdot}_\infty$. This is an arbitrary choice that suffices to make $\DMFTmapDiscretized$ a contracting map on $\HybridizationSpaceDiscretized$ for the range of parameters $(t,u)$ defined by \eqref{eq:AssumptionExistenceTheorem}-\eqref{eq:AssumptionUniquenessTheorem}. 

\begin{proof}[Proof of Theorem~\ref{thm:LocalUniquenessMFDiscretized}]

The first ingredient of the proof is an estimate on the bath-update map.

\begin{lemma}\label{lemma:BUEstimate}
Given $\SelfEnergyVector^1,\SelfEnergyVector^2 \in \LowerHalfPlaneVectors$, the following estimate holds.
\begin{equation} \nonumber
|\BathUpdateDiscretized(\SelfEnergyVector^1)_n-\BathUpdateDiscretized(\SelfEnergyVector^2)_n| \leq  \frac{t^2\mathrm{deg}(\HubbardGraph)}{((2n+1)\pi)^2} |\SelfEnergy_n^1-\SelfEnergy_n^2|.
\end{equation}
\end{lemma}
\begin{proof}
The resolvent identity yields, for all $n \in \IntSubSet{0}{\MatsubaraFrequencyCutoff}$,
\begin{align}
\BathUpdateDiscretized(\SelfEnergyVector^1)_n-\BathUpdateDiscretized(\SelfEnergyVector^2)_n= t^2\left(\SelfEnergy_n^2-\SelfEnergy_n^1\right) w^T\left(i(2n+1)\pi - \SelfEnergy_n^1 - th^0_\perp \right)^{-1} \left(i(2n+1)\pi- \SelfEnergy_n^2 -th^0_\perp \right)^{-1}w, \nonumber
\end{align}
and we have, diagonalizing $h^0_\perp$ similarly as in the proof of Lemma \ref{lemma:BoundednessBUMFDiscretized},
\begin{equation} \nonumber
\Modulus{w^T\left(i(2n+1)\pi - \SelfEnergy_n^1 - th^0_\perp \right)^{-1} \left(i(2n+1)\pi - \SelfEnergy_n^2 -th^0_\perp \right)^{-1}w} \leq \frac{\deg(\HubbardGraph)}{\left((2n+1)\pi\right)^2}, %\leq \frac{\mathrm{deg}(\HubbardGraph)}{\pi^2},
\end{equation}
which concludes the proof.
\end{proof}

In particular, Lemma \ref{lemma:BUEstimate} shows that $\BathUpdateDiscretized$ is Lipschitz continuous with Lipschitz constant $t^2\deg(\HubbardGraph)/\pi^2$ for the distance induced by the maximum norm. Now let $t_{\MatsubaraFrequencyCutoff}>0$ given by Theorem \ref{thm:GlobalExistenceMFDiscretized}. We know that under Assumption~\eqref{eq:AssumptionExistenceTheorem}, that is $t^2\deg(\HubbardGraph)<t_{\MatsubaraFrequencyCutoff}^2$, the compact set $\HybridizationSpaceDiscretized$ is invariant by the map $\DMFTmapDiscretized$. Moreover, the function $F_{\MatsubaraFrequencyCutoff}$ is smooth on $\HybridizationSpaceDiscretized\subset\LowerHalfPlaneVectors$, see Equation \eqref{eq:FunctionCriticalRadius}. Define $L_{\MatsubaraFrequencyCutoff}=\sup\{\Norm{\nabla F_{\MatsubaraFrequencyCutoff}(z)}_\infty\ ; \ z\in\HybridizationSpaceDiscretized \}$, which is then finite.
For all $\Hybridization_1,\Hybridization_2 \in \HybridizationSpaceDiscretized$, we have
\begin{align*}
\Norm{\DMFTmapDiscretized(\Hybridization_1)-\DMFTmapDiscretized(\Hybridization_2)}_\infty & = \Norm{\BathUpdateDiscretized(u^2F_{\MatsubaraFrequencyCutoff}(\Hybridization_1))-\BathUpdateDiscretized(u^2F_{\MatsubaraFrequencyCutoff}(\Hybridization_2))}_\infty   \\
&\leq \left(\frac{tu}{\pi}\right)^2 \mathrm{deg}(\HubbardGraph) \Norm{F_{\MatsubaraFrequencyCutoff}(\Hybridization_1)-F_{\MatsubaraFrequencyCutoff}(\Hybridization_2)}_\infty \\
&\leq \left(\frac{tu}{\pi}\right)^2 \mathrm{deg}(\HubbardGraph) L_{\MatsubaraFrequencyCutoff}\Norm{\Hybridization_1-\Hybridization_1}_\infty,
\end{align*}
and we conclude using Picard fixed-point theorem, with $\eta_{\MatsubaraFrequencyCutoff}=\pi^2/L_{\MatsubaraFrequencyCutoff}$. This constant is of course dependent on the choice of the distance.
\end{proof}

\medskip

\section{The case of bipartite systems}\label{sec:BipartiteSystems}

The purpose of this section is to investigate the special case of bipartite graphs $\HubbardGraph$. Examples of bipartite graphs considered in condensed matter physics include regular trees (also known as Bethe lattices~\cite{georges_dynamical_1996}), the Lieb lattice, and hypercubic lattices of various dimensionalities \cite{georges_dynamical_1996}. Since only {\em infinite} Bethe lattices are vertex-transitive, Hubbard models on Bethe lattices do not fall into the scope of our analysis, which is restricted to finite translation-invariant systems. On the other hand, our results apply to periodic supercell models built from Lieb or hypercubic lattices with an even number of sites in each dimension.

\subsection{Particle-hole symmetry}

As detailed in \cite{lieb_hubbard_2004}, the Hubbard model defined with a bipartite graph $\HubbardGraph$ inherits a specific \emph{particle-hole symmetry}.
More precisely, let $\HubbardVertices=\HubbardVertices_A\sqcup\HubbardVertices_B$ be a partition of the vertices such that 
$$
\{i,j\} \in \HubbardEdges \implies i \in \HubbardVertices_A \mbox{ and } j \in \HubbardVertices_B \mbox{ (or the converse)},
$$
and introduce the corresponding particle-hole transform $\mathcal{B}_{A,B}$ which is the Bogoliubov transform from $\mathcal{L}(\mathrm{Fock}(\OneParticleSpace))$ to itself, defined by
\begin{equation}
    \mathcal{B}_{A,B}(\AnnihilationOperator[i,\sigma])=\left\{
    \begin{matrix}
        \CreationOperator[i,\sigma] \quad &\text{if } i \in A, \\
        -\CreationOperator[i,\sigma] \quad &\text{if } i \in B.
    \end{matrix} \right.
\end{equation}
One can check that the Hubbard Hamiltonian $\HubbardHamiltonian$ (see \cite[Definition 2.5]{cances_mathematical_2024}) is modified under this transformation as
\begin{equation}\label{eq:ParticleHoleTransform}
    \mathcal{B}_{A,B}(\HubbardHamiltonian)=\HubbardHamiltonian + \OnSiteRepulsion[] (\Cardinal{\HubbardVertices}-\TotalNumberOperator),
\end{equation}
so that the grand canonical Hubbard Hamiltonian $\HubbardHamiltonian'=\HubbardHamiltonian-\ChemicalPotential\TotalNumberOperator$ is invariant under this transform for the half-filled case $\ChemicalPotential=\OnSiteRepulsion[]/2$, with $\ChemicalPotential$ denoting the chemical potential. We stick to this case when using the IPT approximation, as discussed more in details in \cite[Section 2.5.2]{cances_mathematical_2024}.

As a consequence, the Gibbs state is also invariant under this transformation: $\mathcal{B}_{A,B}(e^{-\StatisticalTemperature\HubbardHamiltonian'})=e^{-\StatisticalTemperature\HubbardHamiltonian'}$.
From this, one proves that the time-ordered one-body Green's function $G^{\mathrm{TO}}$ associated with the Gibbs state  $Z^{-1} e^{-\beta \HubbardHamiltonian'}$ is anti-hermitian $G^{\mathrm{TO}}=-G^{\mathrm{TO},\dagger}$: indeed, given $t\geq0$ and for all one-particle basis vector $\OneParticleState=\OneParticleState_{i,\sigma}, \text{with } i \in \HubbardVertices, \sigma=\uparrow,\downarrow, $, we have

\begin{align*}
    Z\HermitianProduct{\OneParticleState}{iG^\mathrm{TO}(t)\OneParticleState}&=\Trace\left(e^{-\StatisticalTemperature\HubbardHamiltonian'} e^{it\HubbardHamiltonian'}\AnnihilationOperator[\OneParticleState]e^{-it\HubbardHamiltonian'}\CreationOperator[\OneParticleState]\right)= \overline{\Trace\left(\mathcal{B}_{A,B}(\CreationOperator[\OneParticleState]e^{-it\HubbardHamiltonian'}\AnnihilationOperator[\OneParticleState]e^{it\HubbardHamiltonian'} e^{-\StatisticalTemperature\HubbardHamiltonian'})\right)} \\
    &=\overline{\Trace\left(e^{-\StatisticalTemperature\HubbardHamiltonian'} e^{it\HubbardHamiltonian'}\AnnihilationOperator[\OneParticleState]e^{-it\HubbardHamiltonian'}\CreationOperator[\OneParticleState] \right)} = Z\overline{\HermitianProduct{\OneParticleState}{iG^\mathrm{TO} \OneParticleState}}=-Z\HermitianProduct{\OneParticleState}{iG^{\mathrm{TO},\dagger}\OneParticleState},
\end{align*}
where we make use of the cyclicity of the trace and \eqref{eq:ParticleHoleTransform}. The result for $t<0$ can be proved in a similar way.

Therefore, its Generalized Fourier transform $G$ satisfies the following property: for all $z \in \UpperHalfPlane$,
\begin{equation}
    G(z)=-G(-\bar{z})^\dagger.
\end{equation}

From this, one infers that the Nevanlinna-Riesz measure of $G$ (so to say the spectral function, see \cite[Section 3]{cances_mathematical_2024}) is \emph{symmetric}, or equivalently, that the Green's function $G$ is anti-hermitian on the imaginary axis.

The following corollary ensures that the MaF discretization of the IPT-DMFT equations preserves this property in the following sense: as discussed in Section \ref{sec:IterativeSchemeMottTransiton}, one can extract an approximation $\mathbf{G}=(G_n)_{n \in \IntSubSet{0}{\MatsubaraFrequencyCutoff}} \in \LowerHalfPlaneVectors$ of the local Green's function $G_{1,1}$ at the lowest $(\MatsubaraFrequencyCutoff+1)$ Matsubara frequencies using the constitutive relation
\begin{equation}
 G_n\simeq  G(i\MatsubaraFrequency_n)=(i\MatsubaraFrequency_n-(\Hybridization(i\MatsubaraFrequency_n)+\SelfEnergy(i\MatsubaraFrequency_n)))^{-1}\simeq (i\MatsubaraFrequency_n - (\Hybridization_n+\SelfEnergy_n))^{-1},
\end{equation}
which is also purely imaginary.

\begin{corollary}[MaF-discretized preserves particle-hole symmetry]\label{cor:CoulsonRushbrooke}
Let $\HubbardGraph$ be a \emph{bipartite} vertex transitive graph. Assume that $t$ is small enough, in the sense that it verifies the assumption \eqref{eq:AssumptionExistenceTheorem}. Then the MaF-discretized IPT-DMFT map $\DMFTmapDiscretized$ defined in \eqref{eq:DMFTMapDefinition} preserves purely imaginary vectors:
\begin{equation}
    \HybridizationVector \in \PurelyImaginaryVectors\cap\HybridizationSpaceDiscretized  \implies \DMFTmapDiscretized(\HybridizationVector) \in \PurelyImaginaryVectors\cap\HybridizationSpaceDiscretized .
\end{equation}
In particular, if the assumptions \eqref{eq:AssumptionExistenceTheorem}-\eqref{eq:AssumptionUniquenessTheorem} of Theorem \ref{thm:LocalUniquenessMFDiscretized} are satisfied, the solution $(\HybridizationVector_\star,\SelfEnergyVector_\star )$ is purely imaginary, i.e. $\HybridizationVector_\star,\SelfEnergyVector_\star \in \PurelyImaginaryVectors$.
\end{corollary}

This result is not inherent to this discretization, as a similar result can be derived for the original IPT-DMFT equations. Indeed, the IPT-DMFT map introduced in \cite[Section 3]{cances_mathematical_2024} preserves symmetric measures when the Hubbard graph is bipartite.

\begin{proof}
Let $\HybridizationVector \in \PurelyImaginaryVectors$. One can easily check that $\SelfEnergyVector=u^2F_{\MatsubaraFrequencyCutoff}(\HybridizationVector)$ is purely imaginary. If, furthermore, $\HybridizationVector\in\HybridizationSpaceDiscretized$, with $t$ small enough to satisfy the assumption \eqref{eq:AssumptionExistenceTheorem}, then it has been shown in the proof of Theorem \ref{thm:GlobalExistenceMFDiscretized} that $\SelfEnergyVector\in\LowerHalfPlaneVectors$, so that  $\SelfEnergyVector\in\PurelyImaginaryVectors$. Suppose now that $\HybridizationVector \in \PurelyImaginaryVectors\cap\HybridizationSpaceDiscretized$.

Let $X_n \in \RealNumbers$ be given by $iX_n=i(2n+1)\pi - \SelfEnergy_n$ for $n \in \IntSubSet{0}{\MatsubaraFrequencyCutoff}$. Using the Schur complement methods, we obtain
\begin{equation}
    \DMFTmapDiscretized(\HybridizationVector)_n = \BathUpdateDiscretized(\SelfEnergyVector)_n = iX_n - \left( \left( \left(iX_n- t\AdjacencyMatrix\right)^{-1}\right)_{1,1}\right)^{-1}.
\end{equation}
Now, since $\HubbardGraph$ is bipartite \cite{brouwer2011spectra}, the associated adjacency matrix $\AdjacencyMatrix$ reads
\begin{equation}
    \AdjacencyMatrix=\left(\begin{matrix}
        0& a^T \\
        a & 0 
    \end{matrix} \right)
\end{equation}
and letting $ J=\begin{pmatrix} 1 & 0 \\ 0 & -1 \end{pmatrix}$, we have $J \AdjacencyMatrix J=-\AdjacencyMatrix$.
This implies that $x$ is an eigenvector of $\AdjacencyMatrix$ with eigenvalue $\lambda$ if and only if $Jx$ is also an eigenvector with eigenvalue $-\lambda$. The matrix $\AdjacencyMatrix$ being symmetric real, it is therefore diagonalizable as follow:
\begin{equation}\label{eq:DiagonalizationBipartite}
    \AdjacencyMatrix=PDP^t \text{ with } \left\lbrace\begin{matrix}
        &D=\mathrm{diag}(\lambda_1,\ldots \lambda_p,-\lambda_1,\ldots, -\lambda_p, 0,\ldots,0), \\
        &P=(x_1|\ldots|x_p|Jx_1|\ldots|Jx_p|\ldots).
    \end{matrix}\right.
\end{equation}
Using this diagonalization, we have
\begin{align}
    \left(\left(iX_n-\AdjacencyMatrix\right)^{-1}\right)_{1,1}&=\sum_{j=1}^p \frac{P_{1,j}^2}{iX_n - \lambda_j} + \sum_{j=1}^p \frac{P_{1,j+p}^2}{iX_n+\lambda_j} + \sum_{j=2p+1}^L \frac{P_{1,j}^2}{iX_n} \\
    &=i \left(\sum_{j=1}^p \frac{2P_{1,j}^2X_n}{X_n^2+\lambda_j^2} -\sum_{j=2p+1}^{L}\frac{P_{1,j}^2}{X_n}\right) \in i\RealNumbers,
\end{align}
making use of the fact for all $j \in \IntSubSet{1}{p}, P_{1,j}=x_{j,1}=(JxS_j)_1=P_{1,j+p}$, proving thereby the first claim. Now, if the assumptions \eqref{eq:AssumptionExistenceTheorem}-\eqref{eq:AssumptionUniquenessTheorem} of Theorem \ref{thm:LocalUniquenessMFDiscretized} are satisfied, the sequence $(\HybridizationVector^{(k)})_{k \in \Integers} \in \PurelyImaginaryVectors$ defined by
\begin{equation}
    \HybridizationVector^{(0)} \in \PurelyImaginaryVectors \cap \HybridizationSpaceDiscretized, \quad \forall k \in \Integers,\ \HybridizationVector^{(k+1)}=\DMFTmapDiscretized(\HybridizationVector^{(k)})
\end{equation}
converges toward the unique solution $\HybridizationVector_*$, therefore $\HybridizationVector_* \in \PurelyImaginaryVectors$, and the same holds for the self-energy: $\SelfEnergyVector_* \in \PurelyImaginaryVectors$.
\end{proof}

\subsection{Algebraic structure of MaF-discretized IPT-DMFT for generic bipartite systems}

In this section, we characterize the purely imaginary and physically admissible solutions to the dimensionless MaF-discretized IPT-DMFT equations \eqref{eq:NondimensionalBU}--\eqref{eq:NondimensionalIPT} for bipartite systems. After showing that these equations boil down to a set of sparse algebraic equations for generic bipartite systems, we investigate more in depth the case of the Hubbard dimer.

Seeking the purely imaginary, physically admissible solutions to  \eqref{eq:NondimensionalBU}--\eqref{eq:NondimensionalIPT} amounts to seeking solutions of the form $\HybridizationVector=(-i\pi y_0, -i 3\pi y_1, \cdots , -i(2\MatsubaraFrequencyCutoff+1)\pi y_{\MatsubaraFrequencyCutoff})$, with $\bm y = (y_0,\cdots, y_{\MatsubaraFrequencyCutoff}) \in \R_+^{\MatsubaraFrequencyCutoff+1}$. It is convenient to rewrite this problem in terms of the auxiliary variables 
$$
\bm x = (x_0,\cdots, x_{\MatsubaraFrequencyCutoff}) \quad \mbox{where} \quad x_n := \frac{1}{1+y_n} \quad \mbox{for all } 0 \le n \le \MatsubaraFrequencyCutoff.
$$

For bipartite systems, the spectrum  of $h^0_\perp$ has a special structure. Moreover, removing a vertex from a bipartite graph does not change the bipartite character. The diagonalization is therefore the same as for the total adjacency matrix \eqref{eq:DiagonalizationBipartite}: if $L$ is even, the dimension of $h^0_\perp$ is odd, $0$ is an eigenvalue with odd multiplicity, and the nonzero eigenvalues come in pairs of real eigenvalues with opposite signs. If $L$ is odd, the dimension of $h^0_\perp$ is even and the eigenvalues come in pairs of real eigenvalues with opposite signs ($0$ can be an eigenvalue, but if it is, it has an even multiplicity). In addition if $\nu$ is a positive eigenvalue of $h^0_\perp$, then the orthogonal projectors verify $|\Pi_\nu v|=|\Pi_{-\nu}v|$, with $v=w/|w|$, that is $w=v\sqrt{\deg(\HubbardGraph)}$. Denoting $\Pi_0$ the orthogonal projector on $\ker h^0_\perp$, which is set to zero if $0\notin\sigma(h^0_\perp)$, it follows that \eqref{eq:NondimensionalBU}--\eqref{eq:NondimensionalIPT} can be rewritten as 
$$
\frac{1-x_n}{x_n} = a \deg(\HubbardGraph) \left( \sum_{\substack{\nu \in \sigma(h^0_\perp) \\ \nu > 0}}
\frac{2 |\Pi_\nu v|^2 \left((2n+1)^2- b P_{n,\MatsubaraFrequencyCutoff}(\bm x)\right)}{\left((2n+1)^2- b P_{n,\MatsubaraFrequencyCutoff}(\bm x)\right)^2+(2n+1)^2a \nu^2}
+  \frac{|\Pi_0 v|^2}{(2n+1)^2- b P_{n,\MatsubaraFrequencyCutoff}(\bm x)} \right), 
$$
where the parameters $a$ and $b$ are related to $t$ and $u$ by
$$
a=\frac{t^2}{\pi^2} \quad \mbox{and} \quad b=\frac{u^2}{\pi^4},
$$
and where the homogeneous cubic polynomials $P_{n,\MatsubaraFrequencyCutoff}$, $\MatsubaraFrequencyCutoff \in \N$, $0 \le n \le \MatsubaraFrequencyCutoff$, are given by
\begin{align*}
    P_{n,\MatsubaraFrequencyCutoff}(\bm x) = \! (2n+1) \bigg(& \mathds{1}_{n \ge 1} \!\!\!\! \sum_{\substack{n_1,n_2,n_3=0 \\ n_1 + n_2 + n_3 =n-1}}^{\MatsubaraFrequencyCutoff} \prod_{j=1}^3 \frac{1}{2n_j+1} x_{n_j} \\ &- 3 \!\!\!\!
    \sum_{\substack{n_1,n_2=0 \\ n \le n_1 + n_2 \le \MatsubaraFrequencyCutoff + n}}^{\MatsubaraFrequencyCutoff} \frac{1}{2n_1+1}  \frac{1}{2n_2+1}\frac{1}{2n_1+2n_2-2n+1}  x_{n_1}  x_{n_2} x_{n_1+n_2-n} \\
    & + 3{\mathds 1}_{n \le \MatsubaraFrequencyCutoff-1}  \!\!\!\!\!\!\!\!\! \sum_{\substack{n_1,n_2=0 \\ n_1 + n_2 \le \MatsubaraFrequencyCutoff -(n+1)}}^{\MatsubaraFrequencyCutoff} \!\!\!\!\! \frac{1}{2n_1+1} \frac{1}{2n_2+1}\frac{1}{2n_1+2n_2+2n+3}  x_{n_1}  x_{n_2} x_{n_1+n_2+n+1} \bigg) .
\end{align*}

The above equations read as a system of $(\MatsubaraFrequencyCutoff+1)$ scalar algebraic equations in the $(\MatsubaraFrequencyCutoff+1)$ scalar variables $(x_0,\dots,x_{\MatsubaraFrequencyCutoff})$. Each of these equations can be written as
\begin{equation}\label{eq:polynomial_equation_bipartite}
\sum_{j=0}^{2M+1} (c_{j,n,\MatsubaraFrequencyCutoff}x_n+d_{j,n,\MatsubaraFrequencyCutoff}) P_{n,\MatsubaraFrequencyCutoff}(\bm x)^j = 0,
\end{equation}
where $M$ is the number of distinct positive eigenvalues of $h^0_\perp$, and $c_{j,n,\MatsubaraFrequencyCutoff}$ and $d_{j,n,\MatsubaraFrequencyCutoff}$ are real coefficients, which are themselves polynomials in $a$ and $b$ (or equivalently $t$ and $u$). Note that $M$ is bounded by $\lfloor (L-1)/2 \rfloor$.  As the polynomials $P_{n,\MatsubaraFrequencyCutoff}$ are cubic, homogeneous, the total degree in $\bm x$ of each of these scalar equations is exactly $6M+4$ for any positive values of $a$ and $b$. As the number of monomials in each $P_{n,\MatsubaraFrequencyCutoff}$ is of the order of $\MatsubaraFrequencyCutoff^2$,  the number of monomials in \eqref{eq:polynomial_equation_bipartite} is of order at most $\MatsubaraFrequencyCutoff^{4M+2}$. For large values of $\MatsubaraFrequencyCutoff$, this number is small compared with the total number $\begin{pmatrix} \MatsubaraFrequencyCutoff+6M+5 \\ 6M+4 \end{pmatrix}\sim \MatsubaraFrequencyCutoff^{6M+4}$ of monomials in a generic polynomial equation of degree $(6M+4)$ with $(\MatsubaraFrequencyCutoff+1)$ variables.

\medskip

A bound on the number of solutions for fixed parameters $t,u$ can therefore be obtained using results from symbolic algebraic geometry, while the set of solutions can be explored numerically by homotopy methods from numerical algebraic geometry, although these methods are limited to relatively small values of $N_\omega$.

\subsection{Hubbard dimer}

\label{sec:HubbardDimer}

The Hubbard dimer is the simplest non trivial case. Here $L=2$ and the graph consists of two vertices and one edge. It is clearly bipartite, and we have $\AdjacencyMatrix=\begin{pmatrix}    0 & 1 \\ 1 & 0 \end{pmatrix}$, so that $w=1$ and $h^0_\perp=0$.

\subsubsection{Low-temperature regime for  $\MatsubaraFrequencyCutoff=0$}
\label{sec:Nomega=0}

In this section, we are interested in the simplest case $\MatsubaraFrequencyCutoff=0$, which corresponds to an approximation with only one Matsubara frequency: the unknown is a pair $(\Hybridization,\SelfEnergy) \in -\overline{\UpperHalfPlane}^2$. Combining~\eqref{eq:NondimensionalBU}--\eqref{eq:NondimensionalIPT} with~\eqref{eq:F0}, the problem reduces to finding a solution $\Hybridization \in -\overline{\UpperHalfPlane}$ to the scalar equation
\begin{equation} \label{eq:DimerNomega=0}
    \Hybridization=t^2 \left(i\pi-\frac{3u^2}{\vert i\pi-\Hybridization|^2(i\pi-\Hybridization)}\right)^{-1},
\end{equation}
As such, equation \eqref{eq:DimerNomega=0} can be seen as a polynomial equation in the real variables $(\Re(\Hybridization),\Im(\Hybridization))$ of degree 4.

We would like to investigate the solutions to this equation in the physical setting when the hopping parameter $\HoppingMatrix[]$ and the on-site interaction parameter $\OnSiteRepulsion[]$ are fixed, and the temperature $\StatisticalTemperature^{-1}$ goes to zero. In our dimensionless setting, this corresponds to 
$$
t,u \to + \infty\mbox{  with }   \frac{u}{t}=\frac{\OnSiteRepulsion[]}{\HoppingMatrix[]} =: \alpha > 0 \text{ fixed,}
$$
and it is therefore convenient to rewrite \eqref{eq:DimerNomega=0} as
\begin{equation} \label{eq:DimerNomega=0bis}
    \Hybridization=t^2 \left(i\pi-\frac{3\alpha^2 t^2}{\vert i\pi-\Hybridization|^2(i\pi-\Hybridization)}\right)^{-1}.
\end{equation}

\begin{theorem}
    The solutions to equation \eqref{eq:DimerNomega=0bis} are purely imaginary. The number of solutions that verify $\Hybridization\in-\overline{\ComplexNumbers_+}$ depends on the parameter $\alpha$. In the low-temperature ($\alpha > 0$ fixed, $t \to +\infty$), we have that
    \begin{itemize}
        \item if $0<\alpha<\frac{3\pi}2$, then \eqref{eq:DimerNomega=0bis} has a unique solution $\Hybridization_{\alpha,t}^\infty$ in $\overline{\ComplexNumbers_+}$ and it holds $i\Hybridization_{\alpha,t}^\infty\underset{t\to +\infty}{\longrightarrow}+\infty$;
        \item if $\alpha>\frac{3\pi}2$, then \eqref{eq:DimerNomega=0bis} has exactly three solutions in $\overline{\ComplexNumbers_+}$. One solution goes to infinity with $t$, and the other two solutions converge to distinct nonzero values $\Hybridization_\alpha^1,\Hybridization_\alpha^2\in i\RealNumbers_-$.
    \end{itemize}
    In each case, the branch of solutions escaping to infinity has the following asymptotic behavior:
    \begin{equation}
        \label{eq:AsymptoticBehavior}
        \Hybridization_{\alpha,t}^\infty = -i \frac{t^2}{\pi}+i\frac{3\pi\alpha^2}{t^2}+o\left( \frac{1}{t^{2}} \right).
    \end{equation}
\end{theorem}h

\begin{proof}
Setting 
$$
\epsilon:=t^{-1}, \quad z_\epsilon:=-i\pi^{-1}\Hybridization,  \quad \lambda:= \frac{3\alpha^2}{\pi^2} > 0, 
$$
the problem consists in finding the solutions $z_\epsilon \in \C$ to the equation
\begin{equation} \label{eq:Nomega=0_2}
    z_\epsilon = \frac{(z_\epsilon-1) |z_\epsilon-1|^2}{\lambda - \epsilon^2\pi^2 (z_\epsilon-1) |z_\epsilon-1|^2}  , \quad \mbox{with} \quad {\rm Re}(z_\epsilon) \le 0,
\end{equation}
in the limit $\epsilon \to 0^+$.
Let us first study the limiting equation obtained by setting $\epsilon=0$, namely
\begin{equation} \label{eq:Nomega=0,epsilon=0}
    z_0 = \frac{(z_0-1) |z_0-1|^2}{\lambda}.
    \end{equation}
This equation can be rewritten as
\begin{equation}\label{eq:onz0}
z_0=\frac{|z_0-1|^2}{|z_0-1|^2-\lambda} \in \RealNumbers.
\end{equation}
This shows that any solution $z_0$ to \eqref{eq:Nomega=0,epsilon=0} is real and that any solution to \eqref{eq:Nomega=0,epsilon=0} is therefore a solution to the polynomial equation 
\begin{equation}\label{eq:polynomial_z0}
(z_0-1)^3=\lambda z_0.
\end{equation}
Thus, $x:=z_0-1$ is a real solution to $x^3+px+q=0$ with $p=q=-\lambda < 0$. The discriminant of this equation is $-\lambda^2(27-4\lambda)$. It follows that
\begin{itemize}
\item if $0 < \lambda < \frac{27}4$, then \eqref{eq:polynomial_z0} has one simple real root and two non real conjugated complex roots. For $\lambda=0$, the three roots are equal to $1$. When $\lambda$ increases, there are two branches of complex conjugated roots, and the third root is real and increases. In particular, it is positive;
\item if $\lambda = \frac{27}4$, then the roots of \eqref{eq:polynomial_z0}  are $z_0=4$ (simple root) and $z_0=-1/2$ (double root);
\item if $\lambda > \frac{27}4$, then \eqref{eq:polynomial_z0} has three distinct real roots, one positive greater than 1  as above, and the two other solutions $z_\lambda^1,z_\lambda^2$ are negative, because they behave continuously with $\lambda$, they converge to $-1/2$ when $\lambda\to 27/4$ and they do not cross zero, which is not a solution to \eqref{eq:polynomial_z0}.
\end{itemize}

\medskip

Consider now the case when $\epsilon > 0$.  Recall that we are looking for solutions to~\eqref{eq:Nomega=0_2}, which writes
\begin{equation}\label{eq:z_epsilon}
z_\epsilon = \frac{(z_\epsilon-1) |z_\epsilon-1|^2}{\lambda -  \epsilon^2 \pi^2 (z_\epsilon-1) |z_\epsilon-1|^2}, \quad \mbox{with} \quad {\rm Re}(z_\epsilon) < 0.
\end{equation}
Suppose $z_\epsilon$ is a solution to \eqref{eq:z_epsilon}. Let $b_\epsilon=\frac{\lambda}{|z_\epsilon-1|^2}+\epsilon^2\pi^2-1 \in \RealNumbers$. The discriminant of the quadratic polynomial equation $\epsilon^2 \pi^2 z^2 - b_\epsilon z -1 = 0$ is positive and $z_\epsilon$ is a root of this polynomial. Hence $z_\epsilon$ is real. We are then looking for real solutions to the quartic polynomial equation 
\begin{equation}\label{eq:z_epsilon_2}
 \left( 1+ \epsilon^2\pi^2 z_\epsilon \right) (z_\epsilon-1)^3 =  \lambda z_\epsilon.
\end{equation}
The above equation has four complex roots (counting multiplicities) which depend continuously on the parameter $\epsilon > 0$ (for fixed $\lambda$). As the product of the roots of \eqref{eq:z_epsilon_2} is equal to $- \frac 1{\pi^2\epsilon^2}$ and as the four roots of this polynomial are bounded away from zero uniformly for $0 < \epsilon \le 1$, there is a branch of solutions $\epsilon \mapsto z_{\lambda,\epsilon}^\infty$ that escapes to infinity when $\epsilon \to 0^+$, namely $|z_\epsilon^\infty|\underset{\epsilon\to 0^+}{\longrightarrow} \infty$. Then we necessarily have $(1+\epsilon^2\pi^2 z_{\lambda,\epsilon}^\infty) \to 0$ when $\epsilon \to 0^+$ because of \eqref{eq:z_epsilon_2}. We can thus write $z_{\lambda,\epsilon}^\infty= - \frac{y_\epsilon}{\epsilon^2\pi^2}$ with  $y_\epsilon$ satisfying $y_\epsilon \to 1$ when $\epsilon \to 0^+$ and solution to the quartic polynomial equation 
$$
1 - y_\epsilon = (\epsilon^2\pi^2)^2 \frac{\lambda y_\epsilon}{(y_\epsilon+\epsilon^2 \pi^2)^3} \underset{\epsilon \to 0^+}{\sim} \lambda \pi^4 \epsilon^4.
$$
We finally obtain
\begin{equation} \label{eq:zepsiloninfty}
z_{\lambda,\epsilon}^\infty = - \frac 1{\pi^2\epsilon^2} + \lambda \pi^2 \epsilon^2 + o(\epsilon^2).
\end{equation}
The four branches of solutions therefore behave as follows: one of them, namely $z_\epsilon^\infty$, goes to infinity when $\epsilon \to 0^+$ and behaves as \eqref{eq:zepsiloninfty}, while the other three remain in a compact set at positive distance from the origin, and converge when $\epsilon \to 0^+$ to the three roots to the limiting equation \eqref{eq:polynomial_z0}. We then have to discuss whether these three bounded roots are solutions to \eqref{eq:z_epsilon} or not, depending on the value of $\lambda$.
\begin{itemize}
    \item For $\lambda<27/4$, since the roots of \eqref{eq:polynomial_z0} are either non real or positive. Hence, for $\epsilon$ small enough, the bounded roots of \eqref{eq:z_epsilon_2} are either non real or positive as well, so that they are not real negative. Therefore, they are not solution to \eqref{eq:z_epsilon}. The only solution to \eqref{eq:z_epsilon} is then $z_{\lambda,\epsilon}^\infty$ for $\epsilon$ small enough, which is real, because else its complex conjugate would also be a solution to \eqref{eq:z_epsilon_2}.
    \item For $\lambda>27/4$, the roots of the limiting equation \eqref{eq:polynomial_z0} are three distinct real roots: $z_\lambda^1<z_\lambda^2<0$ and a third root greater than 1. Since the bounded roots of \eqref{eq:z_epsilon_2} are either real or complex conjugated and they converge to the roots of the limiting equation, for $\epsilon$ small enough, they must all be real. Moreover, two of them are real negative and bounded, one of them is positive bounded and the last one is $z_{\lambda,\epsilon}^\infty$, which is real for the same reason as above.
\end{itemize}
\end{proof}

\subsubsection{Purely imaginary solutions for the Hubbard dimer}

Let us conclude this section with the research of purely imaginary solutions in the general case $\MatsubaraFrequencyCutoff\geq 1$, still focusing on the simple case of the Hubbard dimer ($L=2$, $h^0_\perp=0$), for which \eqref{eq:polynomial_equation_bipartite} boils down to solving the system of quartic polynomial equations
\begin{equation}\label{eq:Hubbard-dimer_red}
(x_n-1) \left( (2n+1)^2 - b P_{n,\MatsubaraFrequencyCutoff}(\bm x) \right) + a x_n=0, \quad \forall\ 0 \le n \le \MatsubaraFrequencyCutoff.
\end{equation}
 Whenever $b >0$, that is $u> 0$, each of these polynomial equations contains a number of monomials exactly equal to $(M_{n,\MatsubaraFrequencyCutoff}+1)$, where $M_{n,\MatsubaraFrequencyCutoff}$ is the number of monomials in $P_{n,\MatsubaraFrequencyCutoff}$, so it scales in $O(\MatsubaraFrequencyCutoff^2)$. The system of quartic algebraic equations \eqref{eq:Hubbard-dimer_red} has thus a sparsity of order $\MatsubaraFrequencyCutoff^{-2}$ for large values of~$\MatsubaraFrequencyCutoff$. 

\medskip

As a matter of example of symbolic and numerical algebraic geometry methods, let us consider the simple case when $\MatsubaraFrequencyCutoff=1$. In this case, \eqref{eq:Hubbard-dimer_red} reads: find the solutions $(x_0,x_1)$ in $(0,1]^2$ to the system of quartic equations
\begin{align}\
& (x_0-1)(1-b P_{0,1}(x_0,x_1))+ a x_0=0, \label{eq:system_dimer_Nomega=1_0} \\
& (x_1-1)(9-b P_{1,1}(x_0,x_1))+ a x_1=0, \label{eq:system_dimer_Nomega=1_1}
\end{align}
with 
\begin{align*}
&P_{0,1}(x_0,x_1)= - 3 x_0^3 + x_0^2 x_1 - \frac 23 x_0 x_1^2, \\
&P_{1,1}(x_0,x_1)= 3 x_0^3 - 6 x_0^2 x_1 - \frac 13 x_1^3.
\end{align*}

The dimension of the algebraic variety defined as the set of solutions of \eqref{eq:system_dimer_Nomega=1_0}-\eqref{eq:system_dimer_Nomega=1_1} for a given set of parameters $(a,b) \in \R_+ \times \R_+$ can be computed using methods from symbolic algebraic geometry. The dimension returned by the {\tt Macaulay2} software \cite{M2} is $0$ showing that \eqref{eq:system_dimer_Nomega=1_0}-\eqref{eq:system_dimer_Nomega=1_1} has discrete complex solutions for fixed parameters $(a,b)$. These solutions can be approximated numerically using tools from numerical algebraic geometry. The homotopy methods implemented in the {\tt HomotopyContinuation.jl} Julia package~\cite{HomotopyContinuation.jl} indicate that in the range of parameters $(a,b) \in [0,10] \times [0,25]$, \eqref{eq:system_dimer_Nomega=1_0}-\eqref{eq:system_dimer_Nomega=1_1} have $16$ complex solutions and that the number of admissible solutions satisfying $(x_0,x_1) \in [0,1]^2$ varies from $1$ to $5$ (Figure~\ref{fig:homotopy}). The separation line $b=3(1+a/10)^3$ corresponds to solutions crossing the bottom boundary of the square $(0,1] \times (0,1]$. Indeed, the system \eqref{eq:system_dimer_Nomega=1_0}-\eqref{eq:system_dimer_Nomega=1_1} admits a solution of the form $(x_0,0)$ if and only if $b=3(1+a/10)^3$. The other separation lines correspond to pairs of conjugated solutions merging on the real line in both variables, so that the number of solutions on both sides differs by $2$.

\begin{figure}
    \centering
    \includegraphics[width=0.5\linewidth]{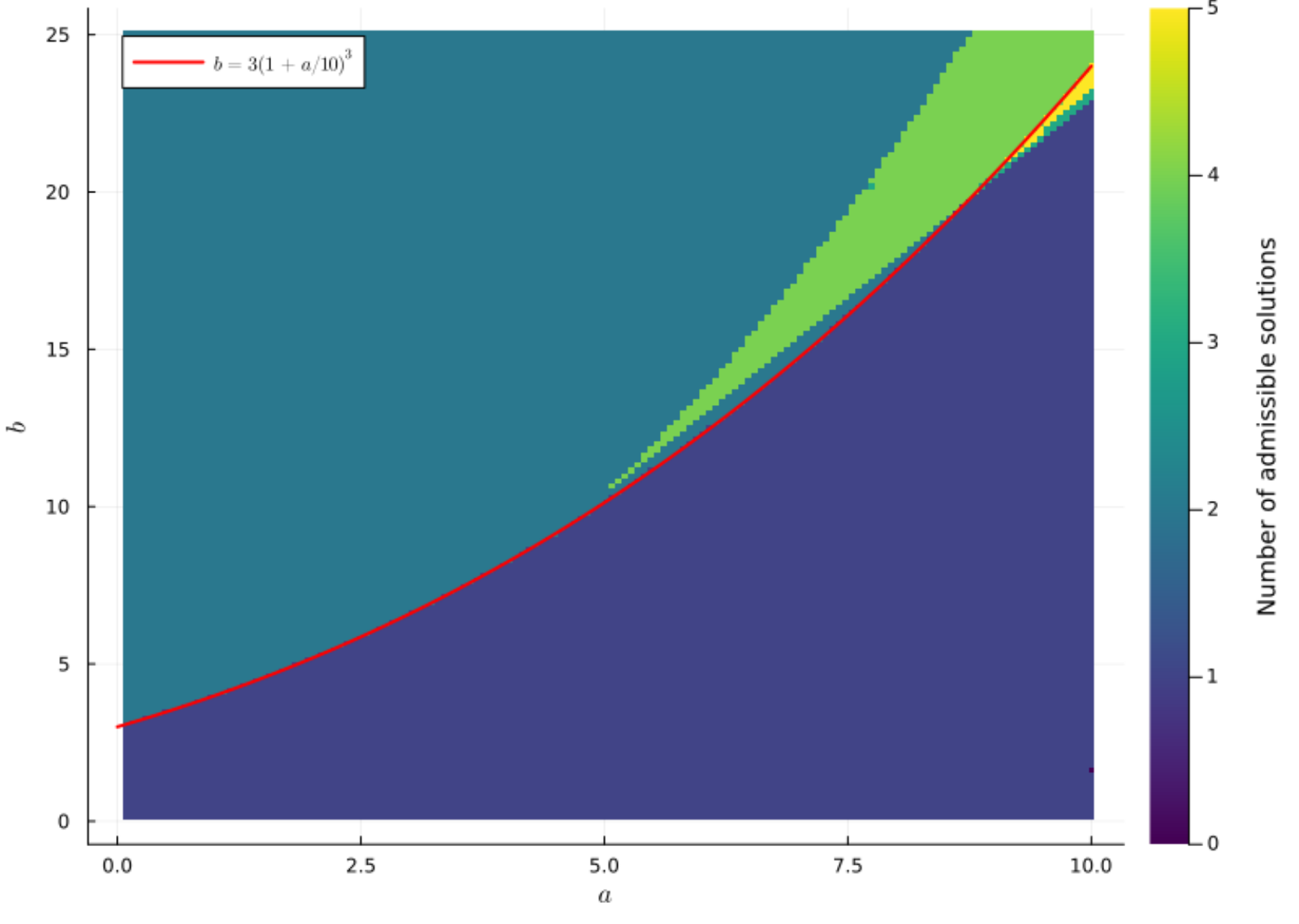}
    \caption{Number of admissible solutions $(x_0,x_1) \in (0,1]^2$ to \eqref{eq:system_dimer_Nomega=1_0}-\eqref{eq:system_dimer_Nomega=1_1} in the range of parameters $(a,b) \in [0,10] \times [0,25]$. This gives the number of purely imaginary solutions to the dimensionless MaF-discretized IPT-DMFT equations \eqref{eq:NondimensionalBU}--\eqref{eq:NondimensionalIPT} for the Hubbard dimer in the range of parameters $(t,u) \in [0,\sqrt{10}\pi] \times [0,5\pi^2]$.}
    \label{fig:homotopy}
\end{figure}

\section{Numerical simulations of the Hubbard dimer}\label{sec:IterativeSchemeMottTransiton}

In this section, we report numerical simulations of the MaF-discretized IPT-DMFT equations for the Hubbard dimer ($\HubbardVertices=\{1,2\}, \HubbardEdges=\{\{1,2\}\}$), the simplest non trivial bipartite and vertex-transitive graph.

In spite of its simplicity, this setting is enough to exhibit numerically an interaction-driven conductor-to-insulator transition, as illustrated in the next section. For this specific purpose, we use the Python/C++ library \texttt{TRIQS} 3.1.0 \cite{parcollet_triqs_2015}, and take advantage of its  analytic continuation functionalities.

Unfortunately, this library is not suited to perform simulations with a small number of Matsubara's frequencies ($\MatsubaraFrequencyCutoff$ needs to be larger than 50 for the ``tail-fitting" procedure to work properly).
In order to exemplify the results of Section \ref{sec:HubbardDimer}, we then switch to a Julia implementation with arbitrary finite-arithmetic precision for the rest of the analysis.

With this implementation, we first perform a series of simulations revealing that the simple fixed point algorithm may converge to a solution that do not fulfill the Pick criterion, meaning that there is no Pick function exactly interpolating the computed values at the lowest $(N_\omega+1)$ lowest Matsubara frequencies. Following on Remark \ref{rmk:AdmissibleSolutions}, this observation indicates that the MaF-discretization do not preserve the fact that $-\Hybridization,-\SelfEnergy$ are Pick functions \cite[Section 3]{cances_mathematical_2024}.

The section is organized as follows: in Section \ref{sec:TRIQS}, we perform a convergence analysis of the algorithm described in Theorem \ref{thm:LocalUniquenessMFDiscretized} using the \texttt{TRIQS} library, while Section \ref{sec:PickCriterion} is devoted to the verification of the Pick criterion for the local Green's function associated to the converged solution.

\subsection{Conductor-to-insulator transition using the \texttt{TRIQS} library}
\label{sec:TRIQS}

In this section, we present numerical evidence of a phase transition for the Hubbard dimer from conductor to insulator in the IPT-DMFT approximation, the simulations being performed using the \texttt{TRIQS} library.

Before going any further, let us specify the meaning of the terms conductor and insulators for the Hubbard dimer. The system is called 
a conductor if the Fermi level (zero in this setting) is in the support of the spectral function~\cite{martin_interacting_2016} , i.e. of the Nevanlinna-Riesz measure of $-G$, and an insulator otherwise~\cite{georges_dynamical_1996}.

To study the conduction properties of a system with the MaF-discretized IPT-DMFT, a strategy is therefore as follows:
\begin{enumerate}
    \item \textbf{Solve the MaF-discretized IPT-DMFT equations}. As a self-consistent method, DMFT equations are usually solved using \emph{an iterative scheme}: for the MaF-discretized IPT-DMFT setting, we implement the simple iterative scheme described in Theorem \ref{thm:LocalUniquenessMFDiscretized}, namely, for a given $N_\mathrm{iter} \in \Integers^{*}$, we compute sequences $\left(\HybridizationVector^{(k)}\right)_{k \in \IntSubSet{0}{N_{\mathrm{iter}}}}$ defined by
    \begin{equation}
        \HybridizationVector^{(0)} \in \HybridizationSpaceDiscretized, \quad \forall k \in \Integers, \quad \HybridizationVector^{(k+1)}= \DMFTmapDiscretized(\HybridizationVector^{(k)}),
    \end{equation}
    for various initial guesses $\HybridizationVector^{(0)}$.
    The number of iterations $N_\mathrm{iter}$ is chosen to ensure that for the considered set of parameters $\HoppingMatrix[],\OnSiteRepulsion[],\StatisticalTemperature$, $\HybridizationVector^{(N_\mathrm{iter})} \approx \HybridizationVector^{(\infty)}$ up to machine precision. We observe that the algorithm converges linearly for the two extreme initial guesses considered: $\HybridizationVector^{(0)}=0$ and $\HybridizationVector_n^{(0)}=1/(i\MatsubaraFrequency_n)$ (Figure \ref{fig:ConvergenceResidualHighT}). These initial values are chosen so that $\HybridizationVector^{(0)} \in \HybridizationSpaceDiscretized$ in both cases and are extreme in the sense that they belong to the boundary of $\HybridizationSpaceDiscretized$. We check that the solution is converged up to machine precision, the latter being fixed by the implementation to $10^{-16}$.
    
    \item \textbf{Compute the MaF-discretized IPT-DMFT local Green's function $\mathbf{\GreensFunction}$}. With the converged solution in the variable $\HybridizationVector$, we extract the associated self-energy $\SelfEnergyVector^{(\infty)}= U^2 \IPTmap_{\StatisticalTemperature,\MatsubaraFrequencyCutoff}(\HybridizationVector^{(\infty)})$. As already mentioned in Section \ref{sec:BipartiteSystems}, we can now extract an approximation $\mathbf{G} \in \LowerHalfPlaneVectors$ of the discretization of the local Green's function $G_{1,1}$, defined by
    \begin{equation}
        \forall n \in \IntSubSet{0}{\MatsubaraFrequencyCutoff},  \quad G_n=\left(i\MatsubaraFrequency_n-\Hybridization_n^{(\infty)}-\SelfEnergy_n^{(\infty)}\right)^{-1}.
    \end{equation}
    \item \textbf{Find an interpolation $G$ of $\mathbf{G}$ and extract the spectral function $A$.} Once a solution is found and that $\mathbf{G}$ is defined, we seek an analytic interpolation $G$ of $\mathbf{G}$ to the whole upper half-plane $\UpperHalfPlane$, namely we solve the following Nevanlinna-Pick interpolation problem \cite{nevanlinna_uber_1919,nicolau_interpolating_1994,nicolau_nevanlinna-pick_2015}
    \begin{equation}\label{eq:NumericalInterpolationProblem}
        \text{Find} -G:\UpperHalfPlane \to \overline{\UpperHalfPlane} \text{ analytic, such that, } \forall n \in \IntSubSet{0}{\MatsubaraFrequencyCutoff}, \quad G(i\MatsubaraFrequency_n)=G_n.
    \end{equation}
    Once $G$ is found (provided a solution to this problem exists, see the following section), we can get the values of the spectral function $A$ by taking the limit of $G$ on the real line, using the Stieljes inversion formula~\cite{gesztesy_matrixvalued_2000}: recall that for $f$ a Pick matrix given by, for all $z \in \UpperHalfPlane$,
    \begin{equation} \nonumber
    f(z)=\int_{\RealNumbers}\frac{1}{\eps-z}d\mu(\eps) 
    \end{equation}
    with $\mu$ a positive matrix-valued Borel finite measure, then we have for all $x,y \in \RealNumbers$,
    \begin{equation}\nonumber
    \frac{1}{2}\mu(\{x\})+ \frac{1}{2}\mu(\{y\}) + \mu((x,y)) = \frac{1}{\pi}\lim_{\eta \to 0^+} \int_x^y\Im(f(\eps+i\eta))d\eps.
    \end{equation}
    In particular, assuming the spectral function $A$ is absolutely continuous with respect to the Lebesgue measure (which seems to hold in the examples below), and letting $\rho$ the associated density, we have
    \begin{equation}
        \rho(\eps)=-\frac{1}{\pi} \lim_{\eta \to 0^+}\Im(\GreensFunction(\eps+i\eta)).
    \end{equation}
    In practice, several algorithms exist to perform this analytic continuation \cite{gubernatis_quantum_1991,jarrell_bayesian_1996,fei_nevanlinna_2021,fei_analytical_2021,huang_robust_2023}, which, although theoretically well founded, suffer from ill-conditioning issues \cite{gubernatis_quantum_1991,jarrell_bayesian_1996}.
    
    For the present purpose, we used the \texttt{set\_from\_pade} algorithm, which is based on Pade approximants, for the reason that it is the most widely used interpolation method in this application field. Given the MaF-discretized Green's function $\mathbf{G}$, it returns the values $\rho_{\RealNumbers}$ of an interpolation of the above points on a given set of points on the real line.
    The latter set is a regular mesh of $[\eps_\mathrm{min},\eps_\mathrm{max}]$ consisting in $N_\mathrm{mesh}$ points.

    Let us mention that, contrary to other methods developed recently in \cite{fei_nevanlinna_2021}, \texttt{set\_from\_pade} does not ensure \emph{a priori} that the analytic continuation is indeed valued in $-\overline{\UpperHalfPlane}$, and may lead to results with negative values of $\rho_\RealNumbers$. This issue is investigated in the next section.
\end{enumerate}

\paragraph{Discretization parameters} Discretization parameters are chosen such that the discretization error due to the integral performed in \eqref{eq:IPTMFDiscretized} can be considered as negligible: we always take a number of points $N_\tau$ associated to the discretization of $(0,\StatisticalTemperature)$ large compared to $\MatsubaraFrequencyCutoff$. Another reason for neglecting numerical quadrature errors is that this integral is actually performed using the \texttt{Fourier} method with tail-fitting (see the \texttt{TRIQS} web site for more details). For all the simulations to come, we took the following parameters:
\begin{equation} \nonumber
N_\tau=1000,\quad \MatsubaraFrequencyCutoff=100 \quad \left(\frac{N_\tau}{\MatsubaraFrequencyCutoff} =10 \right).
\end{equation}

We plot in Figure \ref{fig:MottTransition} the results of this algorithm for the following parameters:
\begin{equation}
\StatisticalTemperature=1,\quad  \HoppingMatrix[]=1, \quad N_{\mathrm{iter}}=50, \quad \HybridizationVector^{(0)}=\left(i\MatsubaraFrequency_n\right)^{-1}_{n \in \IntSubSet{0}{\MatsubaraFrequencyCutoff}},\quad -\eps_{\mathrm{min}}=\eps_{\mathrm{max}}=10, \quad N_\mathrm{mesh}=1000, \label{eq:ParametersMottTransition}
\end{equation}
while varying $\OnSiteRepulsion[]$ from 2 to 10. The big picture is similar to what is obtained for the so-called Bethe lattice \cite{georges_dynamical_1996} and some comments are in order:

\begin{figure}[H]
\centering
\includegraphics[scale=0.8]{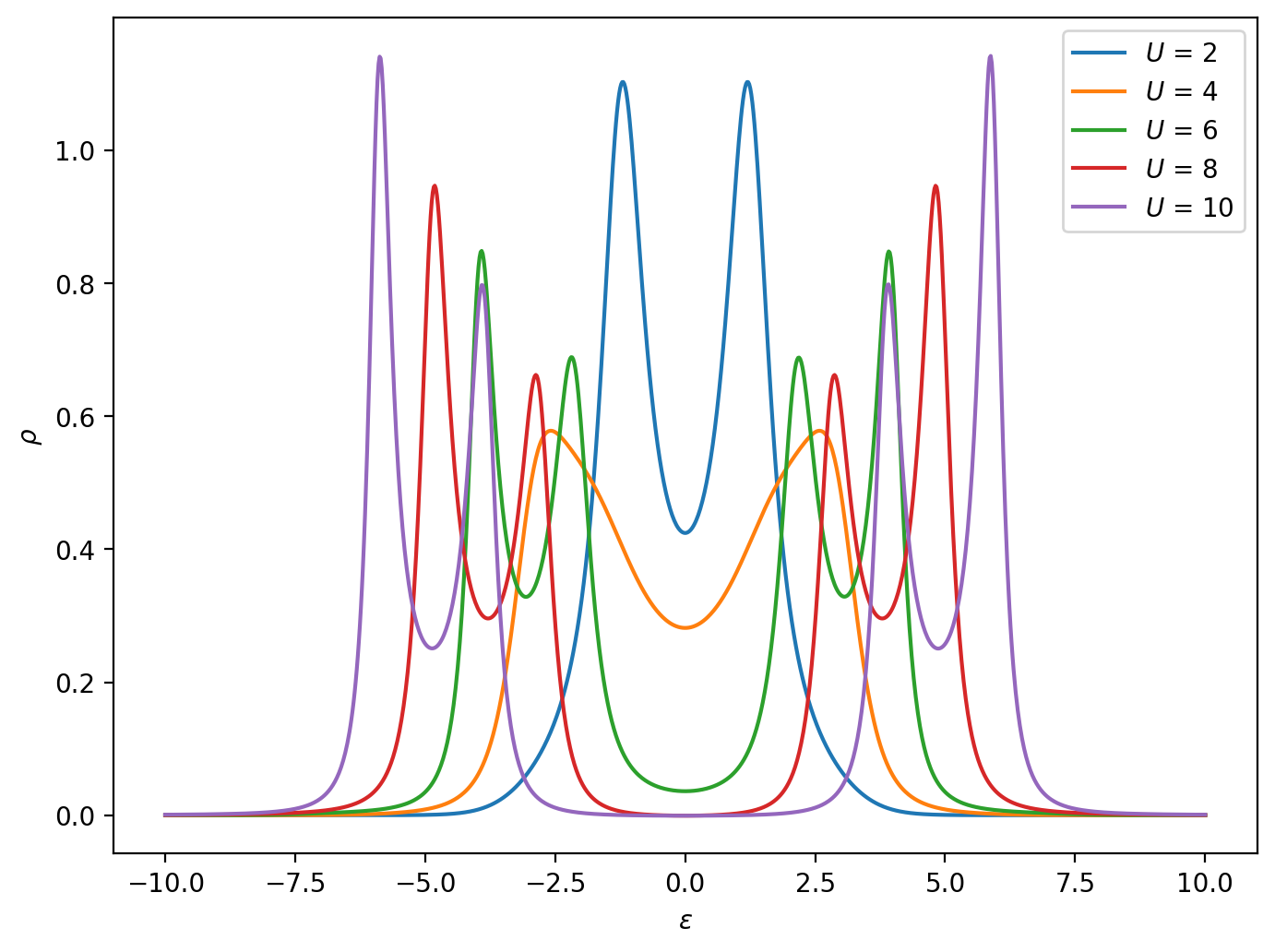}
\caption{Density $\rho$ of the spectral function $A$ obtained by analytic continuation using Pade approximants, for different values of the on-site repulsion $\OnSiteRepulsion[]$. Other parameters are fixed as in \eqref{eq:ParametersMottTransition}.}
\label{fig:MottTransition}
\end{figure}

\begin{itemize}
\item As $\OnSiteRepulsion[]$ increases, $\rho(0)$ decreases to reach approximately $0$ between $\OnSiteRepulsion[]=6$ and $8$: for $\OnSiteRepulsion[] < 6$, the system is a conductor according to the conductivity criterion provided in the introduction of Section \ref{sec:IterativeSchemeMottTransiton}, and an insulator for $\OnSiteRepulsion[]>8$.
\item The density $\rho$ is even, which is consistent with Corollary \ref{cor:CoulsonRushbrooke}: as the initialization $\HybridizationVector^{(0)} \in \PurelyImaginaryVectors$, the iterates are purely imaginary as well: $\forall k \in \Integers, \HybridizationVector^{(k)} \in \PurelyImaginaryVectors$. Note that we observe that this property also holds true at convergence, for non-purely imaginary initializations.
\end{itemize}

\begin{figure}[H]
\begin{subfigure}{0.5\textwidth}
\includegraphics[width=\textwidth]{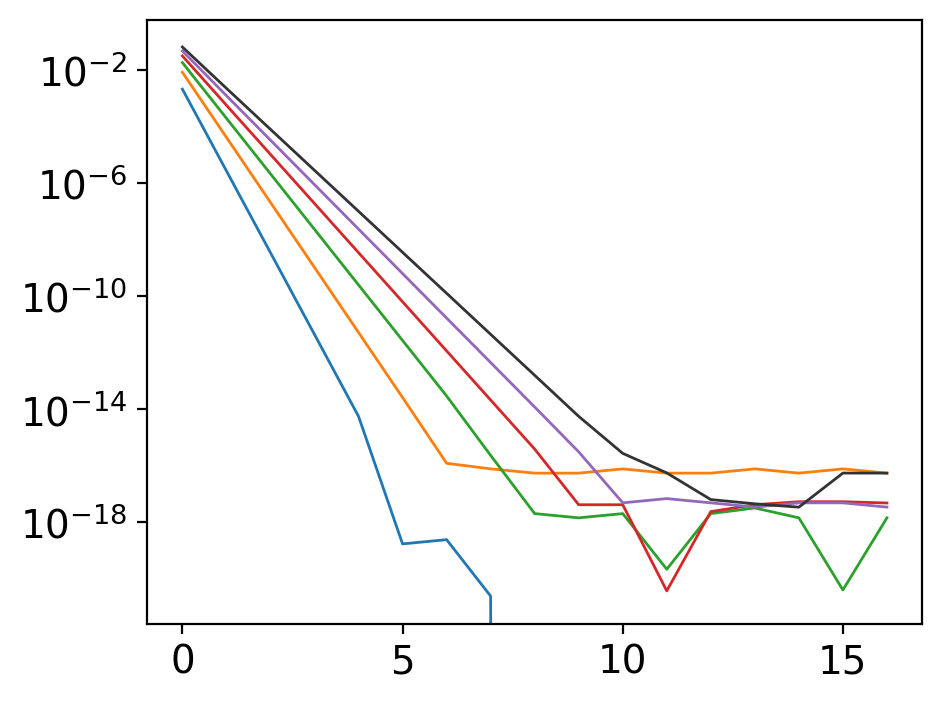}
\includegraphics[width=\textwidth]{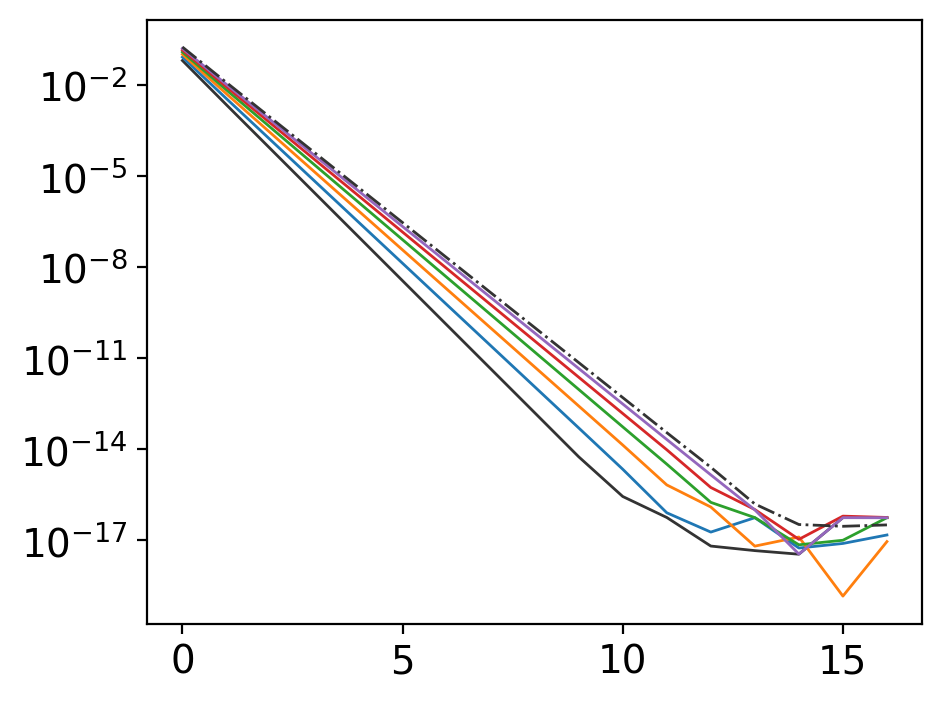}
\includegraphics[width=\textwidth]{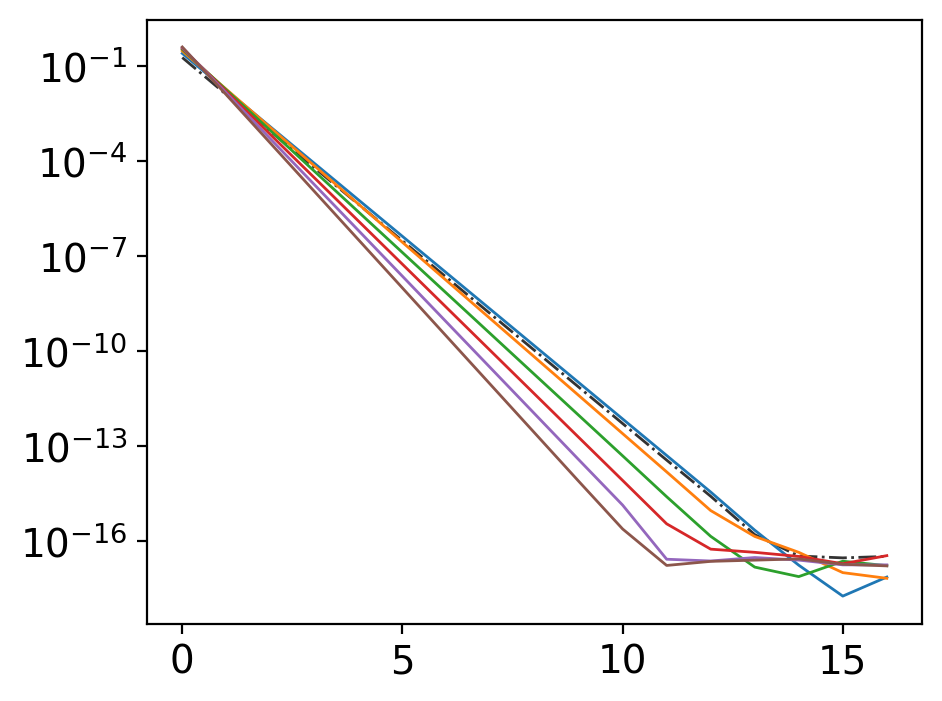}
\subcaption{Initial guess $\Hybridization^{(0)}=\left(i\MatsubaraFrequency_n\right)^{-1}_{n \in \IntSubSet{0}{\MatsubaraFrequencyCutoff}}$}
\end{subfigure}
\begin{subfigure}{0.5\textwidth}
\includegraphics[width=\textwidth]{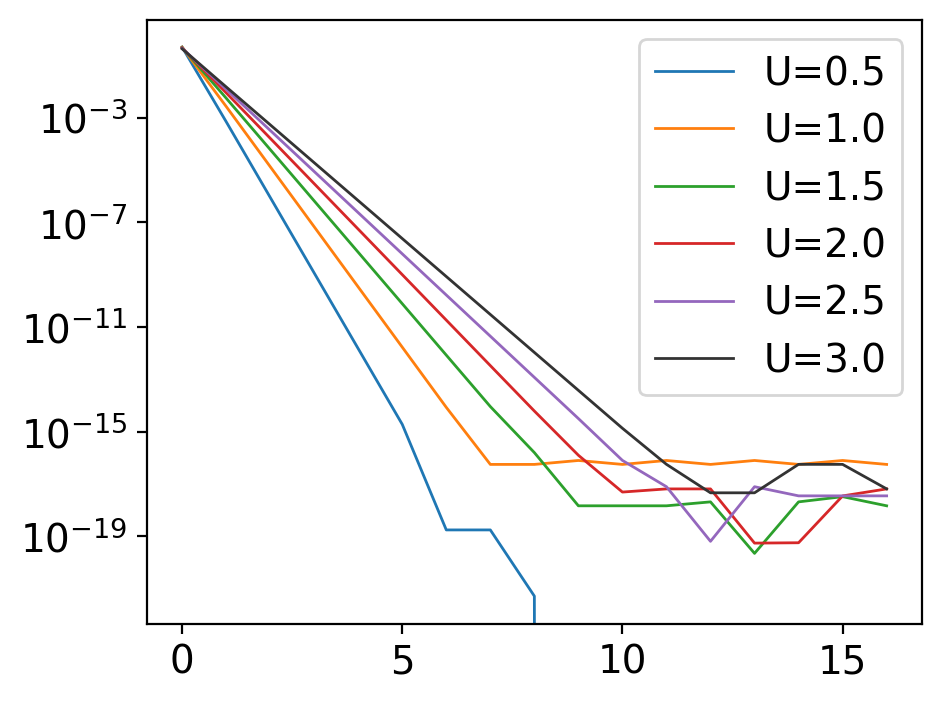}
\includegraphics[width=\textwidth]{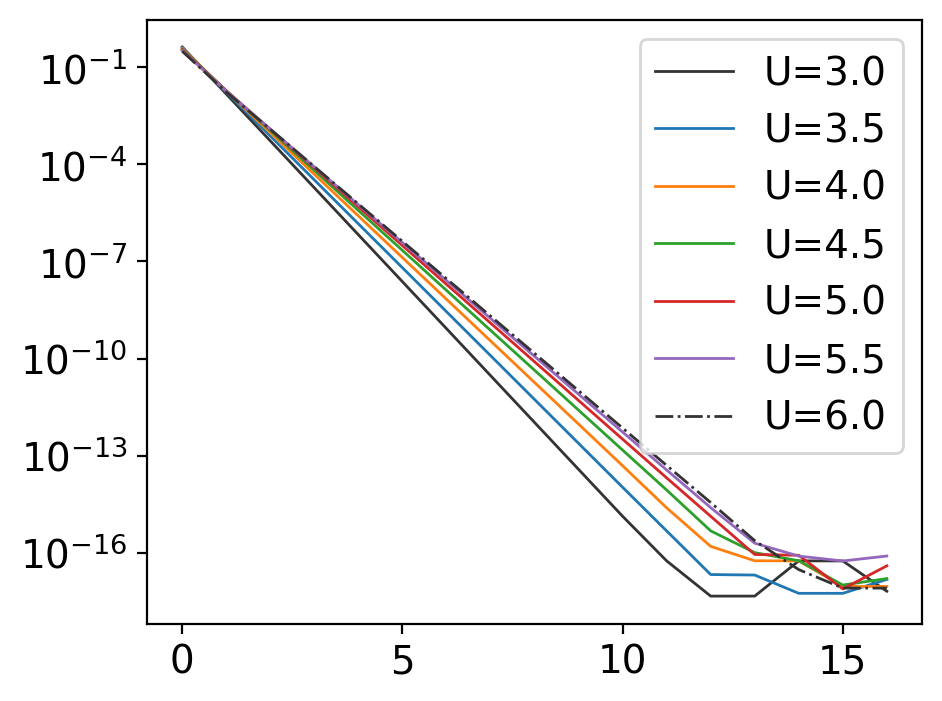}
\includegraphics[width=\textwidth]{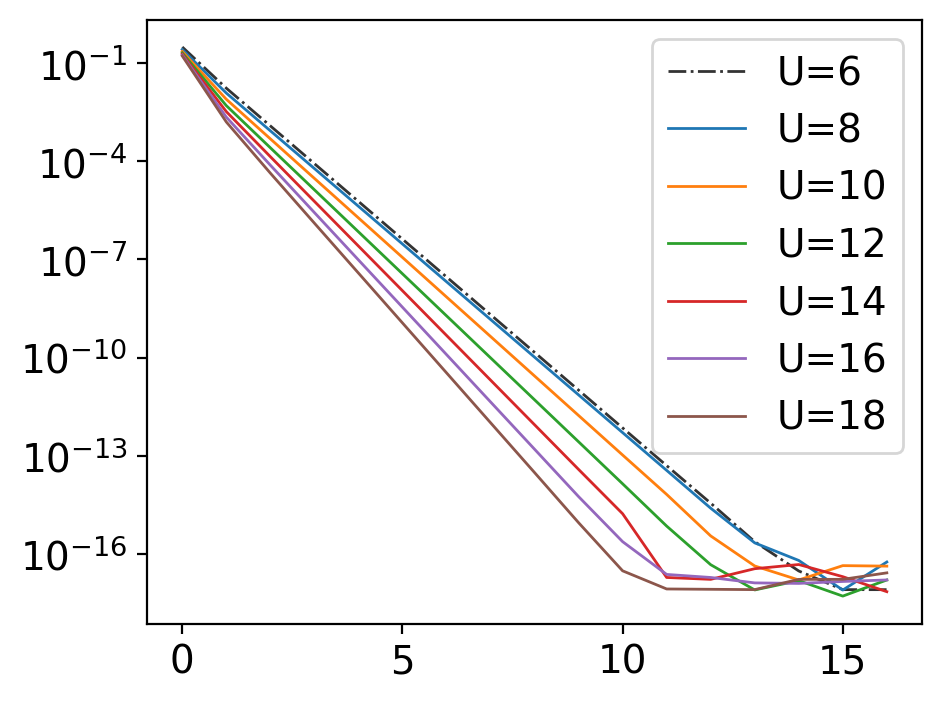}
\subcaption{Initial guess $\Hybridization^{(0)}=0$}
\end{subfigure}
\caption{Linear convergence of the fixed-point algorithm. The residual $\Norm{\Hybridization^{(n+1)}-\Hybridization^{(n)}}_2$ is plotted in log scale, as a function of the iteration $n \in \IntSubSet{0}{N_{\mathrm{iter}}}$. Left and right sides differ in the initial guess $\Hybridization^{(0)}$.}
\label{fig:ConvergenceResidualHighT}
\end{figure}

\begin{figure}
    \centering
    \begin{subfigure}{0.49\textwidth}
        \centering
        \includegraphics[height=6cm]{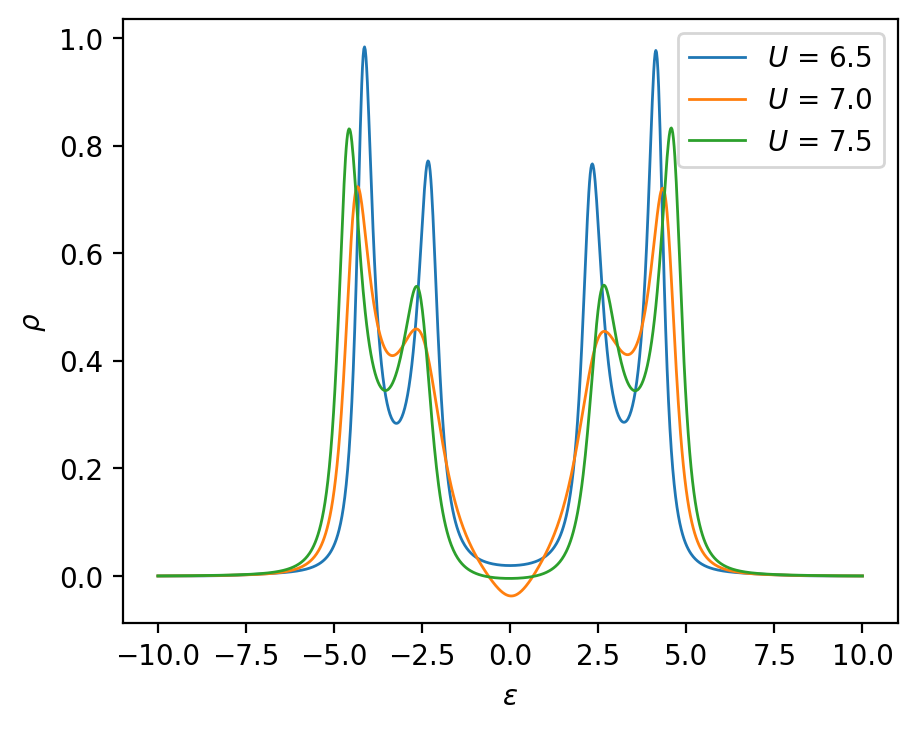}
    \end{subfigure}
    \begin{subfigure}{0.49\textwidth}
        \centering
        \includegraphics[height=6cm]{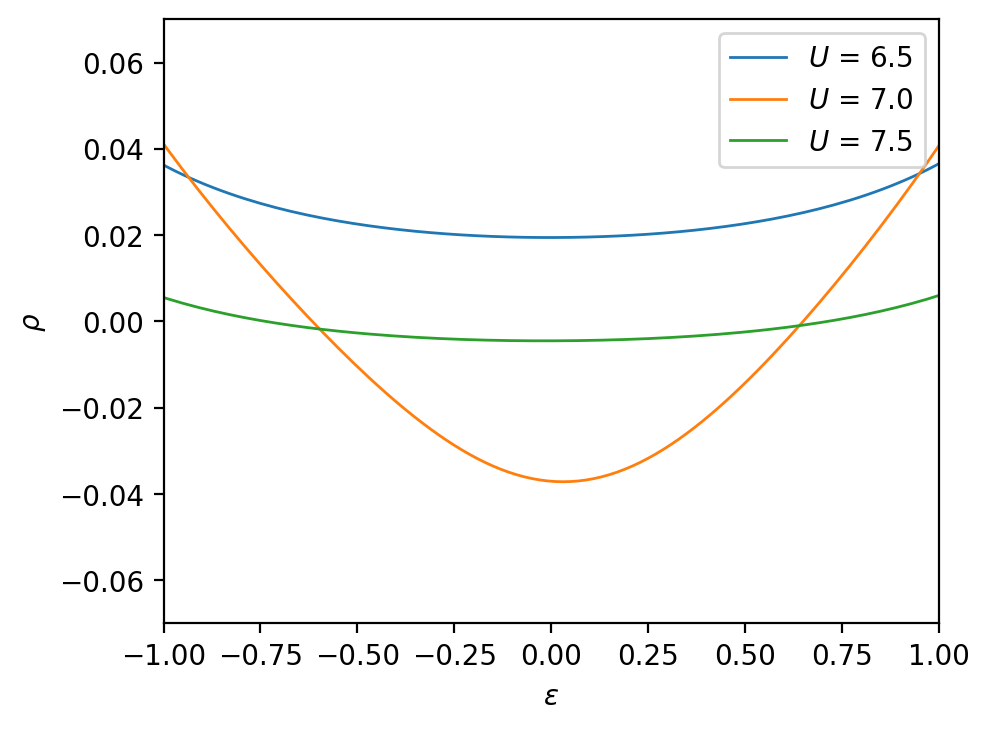}
    \end{subfigure}
    \caption{Density $\rho$ of the spectral function $A$, for different values of the on-site repulsion $U$ (focus around $0$ on the right). For $\OnSiteRepulsion[]=7$ and $\OnSiteRepulsion[]=7.5$, the density $\rho$ takes negatives values around 0.}
    \label{fig:NegativeSpectralFunction}
\end{figure}

Further simulations reported in Figure \ref{fig:NegativeSpectralFunction} indicate that, for some values of $U$ (e.g. $U=7$), \texttt{set\_from\_pade} returns negative values of $\rho(0)$. This observation holds true for larger values of $\MatsubaraFrequencyCutoff$, and is not satisfactory for an approximate solution to \eqref{eq:NumericalInterpolationProblem}. Such an issue may originate from the analytic continuation method we consider, but also from the proper analytic continuation problem defined by the discretized Green's function $\mathbf{\GreensFunction}$. In order to find the origin of this problem, we test numerically the existence of a solution to the Nevanlinna-Pick problem \eqref{eq:NumericalInterpolationProblem} in the next section.

\subsection{Pick criterion on the converged solution}
\label{sec:PickCriterion}

In this section, we are interested in the existence of a solution to the Nevanlinna-Pick problem set in~\eqref{eq:NumericalInterpolationProblem}. Two criteria are available for determining whether such a problem has a solution, and whether this solution is unique (see \cite[Appendix A]{cances_mathematical_2024} and references therein). In this article, we focus on the Pick criterion, which appears to be more robust that the Nevanlinna criterion with respect to small numerical errors.

As explained in details in \cite{nicolau_nevanlinna-pick_2015}, the Nevanlinna-Pick interpolation problem \eqref{eq:NumericalInterpolationProblem} admits a solution if and only if the Pick matrix $P \in \mathcal{S}_{\MatsubaraFrequencyCutoff+1}\left(\ComplexNumbers\right)$ given by
\begin{equation}\label{eq:PickMatrix}
    P_{i,j}=\frac{1-\overline{W(-G_i)}W(-G_j)}{1-\overline{W(i\MatsubaraFrequency_i)}W(i\MatsubaraFrequency_j)}, \text{ with the Cayley transform } W:\left \{ \begin{matrix}
        &\overline{\UpperHalfPlane} \to \UnitDisc=\{z \in \ComplexNumbers, \Modulus{z} \leq 1 \}\\
         &z \mapsto (z-i)/(z+i)
    \end{matrix} \right.,
\end{equation}
is positive semi-definite. The solution is unique if $0$ is an eigenvalue of this matrix, in which case the solution is rational with real poles \cite{donoghue1974interpolation,cances_mathematical_2024}. Testing numerically the existence (and uniqueness) of a solution using the Pick criterion therefore amounts to diagonalize the Pick matrix and to assess the sign of its eigenvalues.

As briefly mentioned in the previous section, the \texttt{TRIQS} library has been developed to run simulations for a large number of Matsubara frequencies. We find in practice that, for $\MatsubaraFrequencyCutoff \leq 50$, the library does not manage to perform the \texttt{Fourier} method, which is used to compute \eqref{eq:def_FNomegabeta}.
Moreover, the latter method uses tail-fitting, which makes the \texttt{TRIQS} implementation of the method slighty different from the one analyzed in Theorem \ref{thm:LocalUniquenessMFDiscretized}.
In addition, the \texttt{TRIQS} library is limited to double-precision computations. In this Section, we are interested in assessing numerically the sign of the lowest eigenvalue of a matrix: to make sure that numerical errors due to the diagonalization do not prevail, we need to use arbitrary-precision finite-arithmetic computations.

For all these reasons, we have used for the rest of Section \ref{sec:IterativeSchemeMottTransiton} our own implementation in \texttt{Julia} of the simple fixed-point algorithm described in \eqref{eq:def_FNomegabeta}.
We use the \texttt{BigFloat} package to handle arbitrary precision. In comparison with the \texttt{TRIQS} implementation, we do not compute numerically the integral defined in \eqref{eq:def_FNomegabeta}: instead, we make use of the alternative expression given in \eqref{eq:NondimensionalIPTOutOfDefinition}, and compute $F_{\MatsubaraFrequencyCutoff}(\HybridizationVector)$ directly. Doing so, our implementation is as close as possible to the simple fixed-point algorithm studied in this article, at the price of greatly reduced performance in terms of computing times.

We proceed as follows:
\begin{enumerate}
    \item \textbf{Solve the MaF-discretized IPT-DMFT equation}. The first step is the same as previously: we implement the simple iterative scheme described in Theorem \ref{thm:LocalUniquenessMFDiscretized}, namely, for a given $N_\mathrm{iter} \in \Integers^{*}$, we compute the sequence $\left(\HybridizationVector^{(k)}\right)_{k \in \IntSubSet{0}{N_{\mathrm{iter}}}}$ by
    \begin{equation}
        \HybridizationVector^{(0)} \in \HybridizationSpaceDiscretized, \quad \forall k \in \Integers, \quad \HybridizationVector^{(k+1)}= \DMFTmapDiscretized(\HybridizationVector^{(k)}).
    \end{equation}
    This time, the number of iterations $N_\mathrm{iter}$ has to be chosen to ensure that for the considered set of parameters $\HoppingMatrix[],\OnSiteRepulsion[],\StatisticalTemperature$, $\HybridizationVector^{(N_\mathrm{iter})} \approx \HybridizationVector^{(\infty)}$ up to a precision that allows us to compute the sign of the lowest eigenvalue of the associated Pick matrix. Our Julia implementation is set to tackle this issue, that is, the precision is chosen so that the numerical error due to the diagonalization of the Pick matrix does not prevent the sign of its smallest eigenvalue from being known exactly.
    
    \item \textbf{Compute the MaF-discretized IPT-DMFT local Green's function $\mathbf{\GreensFunction}$}. As in the previous section, we extract the associated self-energy $\SelfEnergyVector^{(\infty)}= U^2 \IPTmap_{\StatisticalTemperature,\MatsubaraFrequencyCutoff}(\HybridizationVector^{(\infty)})$ and an approximation $\mathbf{G} \in \LowerHalfPlaneVectors$ of the discretization of the local Green's function $G_{1,1}$, defined by
    \begin{equation}\label{eq:LocalGreensFunctionApproximation}
        \forall n \in \IntSubSet{0}{\MatsubaraFrequencyCutoff},  \quad G_n=\left(i\MatsubaraFrequency_n-\Hybridization_n^{(\infty)}-\SelfEnergy_n^{(\infty)}\right)^{-1}.
    \end{equation}
    \item \textbf{Check the Pick criterion.} Define the Pick matrices $P(\Hybridization), P(\SelfEnergy), P(\GreensFunction)$ respectively associated to $-\HybridizationVector,-\SelfEnergyVector$ and $-\mathbf{G}$, using formula \eqref{eq:PickMatrix}. Compute the lowest eigenvalue $\min(\sigma(P(\cdot)))$ of the Pick matrix. If it is negative, then there is no solution to the associated Nevanlinna-Pick interpolation problem.
\end{enumerate}

We report the results of these simulations in Figure \ref{fig:PickCriterion}.  The plots correspond to two different values of $\MatsubaraFrequencyCutoff$. $\MatsubaraFrequencyCutoff=2$ is the lowest value such that there is a range of parameters $(T,U)$ for which the Pick criterion is not satisfied. Note that, because of the structure of \eqref{eq:LocalGreensFunctionApproximation}, if the interpolation problems associated to $\HybridizationVector$ and $\SelfEnergyVector$ have solutions, then so does the problem for $\mathbf{G}$, namely \eqref{eq:NumericalInterpolationProblem}. On the other hand, the converse is not true: as highlighted in Figure \ref{fig:PickCriterion}, it is possible to have a solution for $\mathbf{G}$ and $\HybridizationVector$, but not for $\SelfEnergyVector$.

The plots cover the range or parameters $(T,U)\in[0,10]^2$. The colormap corresponds to the value of $\min(\sigma(P(\cdot)))$, the lowest eigenvalue of the Pick matrix. For a range of parameters ($T$ around 3 for $\HybridizationVector$ and $\mathbf{G}$, $T$ around 5 for $\SelfEnergyVector$, and $U$ around 10), it is clear that the lowest eigenvalue of the Pick matrix is negative. The Pick matrix is not positive semidefinite in this case, hence there is no solution to the Nevanlinna-Pick interpolation problem.

\begin{remark}
    Even if it is not clear on Figure \ref{fig:PickCriterion}, it is numerically observed that when $\MatsubaraFrequencyCutoff$ increases, the range of parameters for which the Pick criterion is satisfied expands. We conjecture that for fixed $(T,U)$, in the limit $\MatsubaraFrequencyCutoff\to\infty$, the Pick criterion is satisfied for $\HybridizationVector$ and $\mathbf{G}$ and that the solutions to the interpolation problem converge to the solutions of the continuous problem exposed in \cite{cances_mathematical_2024}.
\end{remark}

\begin{figure}
    \centering
    \begin{subfigure}{0.49\textwidth}
        \centering
        \includegraphics[width=.9\linewidth]{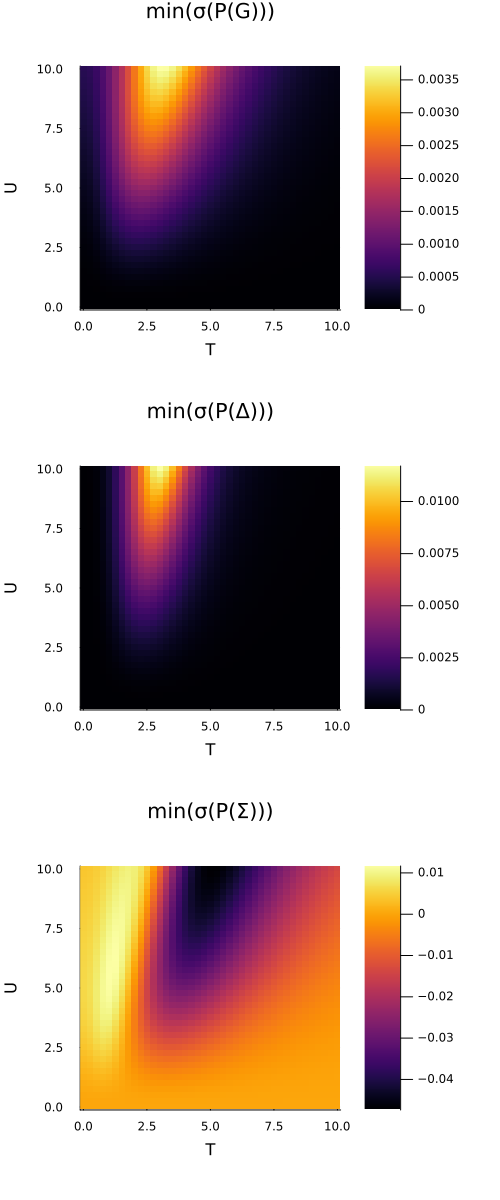}    
        \caption{$\MatsubaraFrequencyCutoff=2$}
    \end{subfigure}
    \begin{subfigure}{0.49\textwidth}
        \centering
        \includegraphics[width=.9\linewidth]{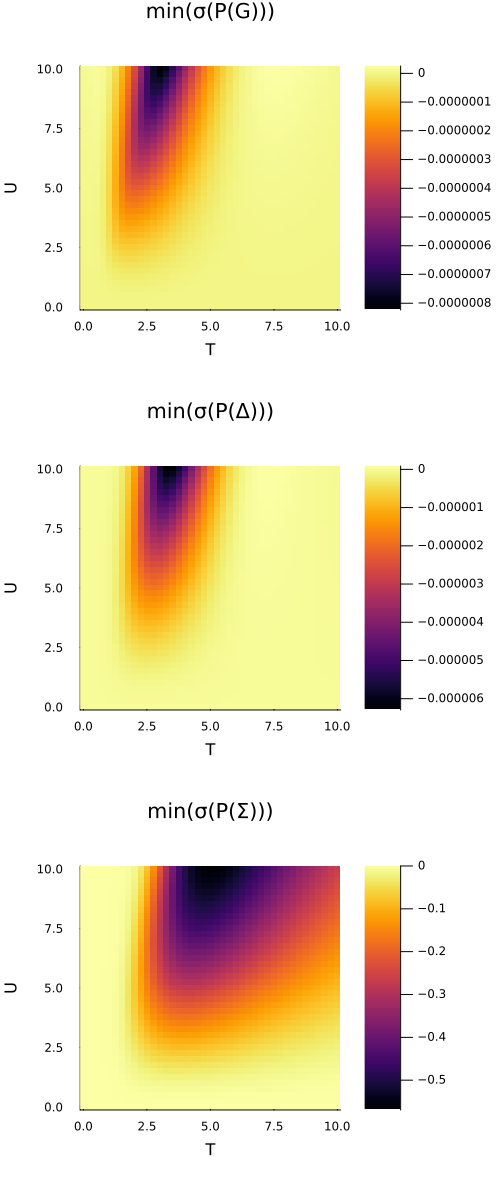}
        \caption{$\MatsubaraFrequencyCutoff=5$}
    \end{subfigure}
    \caption{Test of the Pick criterion (lowest eigenvalue of the Pick matrix) of a solution to the IPT-DMFT equation, for the Nevanlinna-Pick criterion computed for $-\mathbf{G}$ (top), $-\HybridizationVector$ (middle), $-\SelfEnergyVector$ (bottom), for two values of $\MatsubaraFrequencyCutoff$ (left: $\MatsubaraFrequencyCutoff=2$, right: $\MatsubaraFrequencyCutoff=5$) at fixed $\StatisticalTemperature=1$. For a range of parameters $(T,U)$, the solution to the MaF-discretized IPT-DMFT equations leads to analytic interpolation problems that do not have any solution.}
    \label{fig:PickCriterion}
\end{figure}

\newpage

\section*{Acknowledgements} This project has received funding from the Simons Targeted Grant Award No. 896630 and from the European Research Council (ERC) under the European Union's Horizon 2020 research and innovation programme (grant agreement EMC2 No 810367). The authors thank Fabian Faulstich, Michel Ferrero, Antoine Georges, David Gontier, and Fatemeh Mohammadi for useful discussions.

\bibliographystyle{plain}
\bibliography{citations}

\begin{thebibliography}{10}

\bibitem{bravyi_complexity_2017}
Sergey Bravyi and David Gosset.
\newblock Complexity of quantum impurity problems.
\newblock {\em Communications in Mathematical Physics}, 356(2):451--500,
  December 2017.
\newblock arXiv:1609.00735 [cond-mat, physics:math-ph, physics:quant-ph].

\bibitem{HomotopyContinuation.jl}
Paul Breiding and Sascha Timme.
\newblock {H}omotopy{C}ontinuation.jl: {A} {P}ackage for {H}omotopy
  {C}ontinuation in {J}ulia.
\newblock In {\em International Congress on Mathematical Software}, pages
  458--465. Springer, 2018.

\bibitem{brouwer2011spectra}
Andries~E Brouwer and Willem~H Haemers.
\newblock {\em Spectra of graphs}.
\newblock Springer Science \& Business Media, 2011.

\bibitem{cances_mathematical_2016}
Eric Cancès, David Gontier, and Gabriel Stoltz.
\newblock A mathematical analysis of the {GW0} method for computing electronic
  excited energies of molecules.
\newblock {\em Reviews in Mathematical Physics}, 28(04):1650008, 2016.

\bibitem{cances_mathematical_2024}
Éric Cancès, Alfred Kirsch, and Solal Perrin-Roussel.
\newblock A mathematical analysis of {IPT}-{DMFT}, June 2024.
\newblock arXiv:2406.03384.

\bibitem{donoghue1974interpolation}
William~F Donoghue and William~F Donoghue.
\newblock Interpolation by pick functions.
\newblock {\em Monotone Matrix Functions and Analytic Continuation}, pages
  117--127, 1974.

\bibitem{fei_nevanlinna_2021}
Jiani Fei, Chia-Nan Yeh, and Emanuel Gull.
\newblock Nevanlinna {Analytical} {Continuation}.
\newblock {\em Physical Review Letters}, 126(5):056402, February 2021.

\bibitem{fei_analytical_2021}
Jiani Fei, Chia-Nan Yeh, Dominika Zgid, and Emanuel Gull.
\newblock Analytical continuation of matrix-valued functions: {Carathéodory}
  formalism.
\newblock {\em Physical Review B}, 104(16):165111, 2021.

\bibitem{georges_hubbard_1992}
Antoine Georges and Gabriel Kotliar.
\newblock Hubbard model in infinite dimensions.
\newblock {\em Physical Review B}, 45(12):6479--6483, March 1992.

\bibitem{georges_dynamical_1996}
Antoine Georges, Gabriel Kotliar, Werner Krauth, and Marcelo~J. Rozenberg.
\newblock Dynamical mean-field theory of strongly correlated fermion systems
  and the limit of infinite dimensions.
\newblock {\em Reviews of Modern Physics}, 68(1):13--125, January 1996.

\bibitem{gesztesy_matrixvalued_2000}
Fritz Gesztesy and Eduard Tsekanovskii.
\newblock On {Matrix}–{Valued} {Herglotz} {Functions}.
\newblock {\em Mathematische Nachrichten}, 218(1):61--138, 2000.

\bibitem{M2}
Daniel~R. Grayson and Michael~E. Stillman.
\newblock Macaulay2, a software system for research in algebraic geometry.
\newblock Available at \url{http://www2.macaulay2.com}.

\bibitem{gubernatis_quantum_1991}
J.~E. Gubernatis, Mark Jarrell, R.~N. Silver, and D.~S. Sivia.
\newblock Quantum {Monte} {Carlo} simulations and maximum entropy: {Dynamics}
  from imaginary-time data.
\newblock {\em Physical Review B}, 44(12):6011--6029, 1991.

\bibitem{gull_continuous-time_2011}
Emanuel Gull, Andrew~J. Millis, Alexander~I. Lichtenstein, Alexey~N. Rubtsov,
  Matthias Troyer, and Philipp Werner.
\newblock Continuous-time {Monte} {Carlo} methods for quantum impurity models.
\newblock {\em Reviews of Modern Physics}, 83(2):349--404, 2011.

\bibitem{gutzwiller_effect_1963}
Martin~C. Gutzwiller.
\newblock Effect of {Correlation} on the {Ferromagnetism} of {Transition}
  {Metals}.
\newblock {\em Physical Review Letters}, 10(5):159--162, March 1963.

\bibitem{held_mott-hubbard_2001}
Karsten Held, Georg Keller, Volker Eyert, Dieter Vollhardt, and Vladimir~I
  Anisimov.
\newblock Mott-{Hubbard} {Metal}-{Insulator} {Transition} in {Paramagnetic}
  {V}\$\_2\$ {O}\$\_3\$: {An} {LDA}+{DMFT} ({QMC}) {Study}.
\newblock {\em Physical review letters}, 86(23):5345, 2001.

\bibitem{huang_robust_2023}
Zhen Huang, Emanuel Gull, and Lin Lin.
\newblock Robust analytic continuation of {Green}'s functions via projection,
  pole estimation, and semidefinite relaxation.
\newblock {\em Physical Review B}, 107(7):075151, February 2023.
\newblock arXiv:2210.04187 [cond-mat, physics:physics].

\bibitem{hubbard_electron_1963}
John Hubbard.
\newblock Electron correlations in narrow energy bands.
\newblock {\em Proceedings of the Royal Society of London. Series A.
  Mathematical and Physical Sciences}, 276(1365):238--257, 1963.

\bibitem{jarrell_bayesian_1996}
Mark Jarrell and J.~E. Gubernatis.
\newblock Bayesian inference and the analytic continuation of imaginary-time
  quantum {Monte} {Carlo} data.
\newblock {\em Physics Reports}, 269(3):133--195, May 1996.

\bibitem{kanamori_electron_1963}
Junjiro Kanamori.
\newblock Electron {Correlation} and {Ferromagnetism} of {Transition} {Metals}.
\newblock {\em Progress of Theoretical Physics}, 30(3):275--289, 1963.

\bibitem{li_interaction-expansion_2022}
Jia Li, Yang Yu, Emanuel Gull, and Guy Cohen.
\newblock Interaction-expansion inchworm {Monte} {Carlo} solver for lattice and
  impurity models.
\newblock {\em Physical Review B}, 105(16):165133, 2022.

\bibitem{lieb_hubbard_2004}
Elliott~H. Lieb.
\newblock The {Hubbard} model: {Some} {Rigorous} {Results} and {Open}
  {Problems}.
\newblock In {\em Condensed {Matter} {Physics} and {Exactly} {Soluble}
  {Models}: {Selecta} of {Elliott} {H}. {Lieb}}. Berlin, Heidelberg, 2004.

\bibitem{lin_sparsity_2020}
Lin Lin and Michael Lindsey.
\newblock Sparsity {Pattern} of the {Self}-energy for {Classical} and {Quantum}
  {Impurity} {Problems}.
\newblock {\em Annales Henri Poincaré}, 21(7):2219--2257, 2020.

\bibitem{martin_interacting_2016}
Richard~M. Martin, Lucia Reining, and David~M. Ceperley.
\newblock {\em Interacting {Electrons}: {Theory} and {Computational}
  {Approaches}}.
\newblock Cambridge University Press, Cambridge, 2016.

\bibitem{metzner_correlated_1989}
Walter Metzner and Dieter Vollhardt.
\newblock Correlated {Lattice} {Fermions} in d = \${\textbackslash}infty\$
  {Dimensions}.
\newblock {\em Physical Review Letters}, 62(3):324--327, January 1989.

\bibitem{nevanlinna_uber_1919}
Rolf Nevanlinna.
\newblock Uber beschrankte {Funktionen}, die in gegeben {Punkten}
  vorgeschrieben {Werte} annehmen.
\newblock {\em Ann. Acad. Sci. Fenn. Ser. A 1 Mat. Dissertationes}, 1919.

\bibitem{nicolau_interpolating_1994}
Artur Nicolau.
\newblock Interpolating {Blaschke} products solving {Pick}-{Nevanlinna}
  problems.
\newblock {\em Journal d’Analyse Mathematique}, 62(1):199--224, December
  1994.

\bibitem{nicolau_nevanlinna-pick_2015}
Artur Nicolau.
\newblock The {Nevanlinna}-{Pick} {Interpolation} {Problem}.
\newblock 2015.

\bibitem{parcollet_triqs_2015}
Olivier Parcollet, Michel Ferrero, Thomas Ayral, Hartmut Hafermann, Igor
  Krivenko, Laura Messio, and Priyanka Seth.
\newblock {TRIQS}: {A} toolbox for research on interacting quantum systems.
\newblock {\em Computer Physics Communications}, 196:398--415, 2015.

\bibitem{pariser_semiempirical_1953}
Rudolph Pariser and Robert~G. Parr.
\newblock A {Semi}‐{Empirical} {Theory} of the {Electronic} {Spectra} and
  {Electronic} {Structure} of {Complex} {Unsaturated} {Molecules}. {II}.
\newblock {\em The Journal of Chemical Physics}, 21(5):767--776, 1953.

\bibitem{pick_uber_1915}
Georg Pick.
\newblock Über die {Beschränkungen} analytischer {Funktionen}, welche durch
  vorgegebene {Funktionswerte} bewirkt werden.
\newblock {\em Mathematische Annalen}, 77(1):7--23, March 1915.

\bibitem{pople_electron_1953}
John~A. Pople.
\newblock Electron interaction in unsaturated hydrocarbons.
\newblock {\em Transactions of the Faraday Society}, 49(0):1375--1385, 1953.

\bibitem{rubtsov_continuous-time_2005}
A.~N. Rubtsov, V.~V. Savkin, and A.~I. Lichtenstein.
\newblock Continuous-time quantum {Monte} {Carlo} method for fermions.
\newblock {\em Physical Review B}, 72(3):035122, July 2005.

\bibitem{seth_triqscthyb_2016}
Priyanka Seth, Igor Krivenko, Michel Ferrero, and Olivier Parcollet.
\newblock {TRIQS}/{CTHYB}: {A} continuous-time quantum {Monte} {Carlo}
  hybridisation expansion solver for quantum impurity problems.
\newblock {\em Computer Physics Communications}, 200:274--284, 2016.

\bibitem{shapiro_fixed-point_2016}
Joel~H. Shapiro.
\newblock {\em A {Fixed}-{Point} {Farrago}}.
\newblock Universitext. Springer International Publishing, Cham, 2016.

\bibitem{sun_extended_2002}
Ping Sun and Gabriel Kotliar.
\newblock Extended dynamical mean-field theory and {GW} method.
\newblock {\em Physical Review B}, 66(8):085120, 2002.

\bibitem{tan_doping_2023}
Yuting Tan, Pak Ki~Henry Tsang, Vladimir Dobrosavljević, and Louk Rademaker.
\newblock Doping a {Wigner}-{Mott} insulator: {Exotic} charge orders in
  transition metal dichalcogenide moiré heterobilayers.
\newblock {\em Physical Review Research}, 5(4):043190, 2023.

\bibitem{werner_continuous-time_2006}
Philipp Werner, Armin Comanac, Luca de’ Medici, Matthias Troyer, and
  Andrew~J. Millis.
\newblock Continuous-{Time} {Solver} for {Quantum} {Impurity} {Models}.
\newblock {\em Physical Review Letters}, 97(7):076405, 2006.

\bibitem{werner_hybridization_2006}
Philipp Werner and Andrew~J. Millis.
\newblock Hybridization expansion impurity solver: {General} formulation and
  application to {Kondo} lattice and two-orbital models.
\newblock {\em Physical Review B}, 74(15):155107, 2006.

\bibitem{zhang_mott_1993}
X.~Y. Zhang, M.~J. Rozenberg, and G.~Kotliar.
\newblock Mott transition in the d ={\textbackslash}infty {Hubbard} model at
  zero temperature.
\newblock {\em Physical Review Letters}, 70(11):1666--1669, March 1993.

\end{thebibliography}

\end{document}